\documentclass[11pt]{amsart}

\usepackage[text={455pt,670pt},centering]{geometry}
\usepackage{color}
\usepackage{amssymb}
\usepackage{graphicx}
\usepackage{enumitem}
\usepackage{tikz}
\usepackage{srcltx}
\usepackage{hyperref}
\usepackage{dsfont}

\newtheorem{theorem}{Theorem}
\newtheorem{proposition}[theorem]{Proposition}
\newtheorem{lemma}[theorem]{Lemma}

\theoremstyle{definition}
\newtheorem{remark}[theorem]{Remark}

\newtheorem{definition}[theorem]{Definition}

\numberwithin{equation}{section}
\numberwithin{figure}{section}
\numberwithin{theorem}{section}

\newcommand{\Z}{\mathbb{Z}}
\newcommand{\N}{\mathbb{N}}
\newcommand{\R}{\mathbb{R}}
\newcommand{\RR}{\mathbb{R}}

\newcommand{\abs}[1]{\left| #1 \right|}

\DeclareMathOperator{\sgn}{sgn}

\renewcommand{\tilde}{\widetilde}

\newcommand{\eps}{\varepsilon}
\newcommand{\pa}{\partial}

\definecolor{refkey}{gray}{0.75}
\colorlet{labelkey}{blue} 
\begin{document}
\title[Synchronising and non-synchronising dynamics]{Synchronising and Non-synchronising dynamics for a two-species  aggregation model}

\begin{abstract}
This paper deals with analysis and numerical simulations of a one-dimensional two-species hyperbolic aggregation model.
This model is formed by a system of transport equations with nonlocal velocities, which describes the aggregate dynamics of a two-species population in interaction appearing for instance in bacterial chemotaxis.
Blow-up of classical solutions occurs in finite time. This raises the question to define measure-valued solutions for this system.
To this aim, we use the duality method developed for transport equations with discontinuous velocity to prove the existence and uniqueness of measure-valued solutions.
The proof relies on a stability result. In addition, this approach allows to study the hyperbolic limit of a kinetic chemotaxis model.
Moreover, we propose a finite volume numerical scheme whose convergence towards measure-valued solutions
is proved. It allows for numerical simulations capturing the behaviour after blow up.
Finally, numerical simulations illustrate the complex dynamics of aggregates until the formation of a single aggregate: after blow-up of classical solutions, aggregates of different species 
are synchronising or nonsynchronising when collide, that is move together or separately, depending on the parameters of the model and masses of species involved.
\end{abstract}

\author[C. Emako-Kazianou]{Casimir Emako-Kazianou}
\address{Sorbonne Universit\'es, UPMC Univ Paris 06, UMR 7598, Laboratoire Jacques-Louis Lions, F-75005, Paris, France\\
CNRS, UMR 7598, Laboratoire Jacques-Louis Lions, F-75005, Paris, France \\
INRIA-Paris-Rocquencourt, EPC MAMBA, Domaine de Voluceau, BP105, 78153 Le Chesnay Cedex}
\email{emako@ann.jussieu.fr}

\author[J.Liao]{Jie Liao}
\address{Department of Mathematics, East China University of Science and Technology,  Shanghai, 200237, P. R. China}
\email{liaojie@ecust.edu.cn}

\author[N. Vauchelet]{Nicolas Vauchelet}
\address{Sorbonne Universit\'es, UPMC Univ Paris 06, UMR 7598, Laboratoire Jacques-Louis Lions, F-75005, Paris, France \\
CNRS, UMR 7598, Laboratoire Jacques-Louis Lions, F-75005, Paris, France \\
INRIA-Paris-Rocquencourt, EPC MAMBA, Domaine de Voluceau, BP105, 78153 Le Chesnay Cedex}
\email{vauchelet@ann.jussieu.fr}

\date{\today}

\maketitle

\section{Introduction}
Aggregation phenomena for a population of individuals interacting through an interaction potential are usually modelled by the so-called aggregation equation which 
is a nonlocal nonlinear conservation equation. This equation governs the dynamics of the density of individuals subject to an interaction potential $K$.
In this work, we are interested in the case where the population consists of two species which respond to the interaction potential in different ways.
In the one-dimensional case, the system of equations writes:
\begin{equation}\label{aggregation_eq}
 \partial_t \rho_{\alpha}+\chi_{\alpha} \partial_x \left(a(\rho)\rho_{\alpha}\right)=0,\quad \text{for } \alpha=1, 2,
\end{equation}
with 
\begin{equation*}
 a(\rho):=\int_{\R} \partial_x K(x-y)\rho(t,dy),\quad \rho:=\theta_1 \rho_1+\theta_2 \rho_2,
\end{equation*}
where $\theta_{\alpha},\chi_{\alpha}$ for $\alpha=1,2$ are positive constants.

In this work, we are interested in the case where the interaction potential $K$ in
\eqref{aggregation_eq} is {\it pointy} i.e. satisfies the following assumptions:
\begin{enumerate}[label=\textup{ }{(H\arabic*)}\textup{ },ref={(H\arabic*)}]
\item \label{h1} $K \in C^1(\R \backslash \{0\})$.
\item \label{h2} $\forall x \in \R,\quad  K(x)=K(-x)$.
\item \label{h3} $\partial_x K \in L^{\infty}(\R)$.
\item \label{h4} $K$ is $\lambda$-concave with $\lambda>0$ i.e., 
 \begin{equation*}
  \forall x,y \in \R^{*},\quad \left(\partial_x K(x)-\partial_x K(y) \right)(x-y)\leq \lambda (x-y)^2.
 \end{equation*}
\end{enumerate}
The aggregation equation arises in several applications in biology and physics. In fact, it is encountered in the modelling of cells which move in response to chemical cues.
The velocity of cells $a(\rho)$ depending on the distribution of nearby cells represents the gradient of the chemical substance which triggers the motion.
Cells gather and form accumulations near regions more exposed to oxygen as observed in \cite{Mittal,VCJS}. We can also describe the movement of pedestrians using the aggregation equation as in \cite{Helbing} 
where the velocity of pedestrians is influenced by the distribution of neighbours. This equation can also be applied to model opinion formation (see \cite{Sznajd}) where interactions between different opinions 
can be expressed by a convolution with the kernel $K$.

From the mathematical point of view, it is known that solutions to the aggregation equation with a pointy potential blow up in finite time (see e.g \cite{Francesco,Carrillo,Bertozzi}).
Then global-in-time existence for weak measure solutions has been investigated. In \cite{Carrillo}, existence of weak solutions for single species model has been obtained as a gradient flow.
This technique has been extended to the two-species model at hand in \cite{Francesco}. Another approach of defining weak solution for such kind of model has been proposed in \cite{jamesnv,GF_dual} for the single species case.
In this approach, the aggregation equation is seen as a transport equation with a discontinuous velocity $a(\rho)$.
Then solutions in the sense of duality have been defined for the aggregation equation.

Duality solutions has been introduced for linear transport equations with discontinuous velocity in the one-dimensional space in \cite{Bouchut}.
Then it has been adapted to the study of nonlinear transport equations in \cite{JamesBouchut,jamesnv,GF_dual}. 
In \cite{jamesnv,GF_dual}, the authors use this notion of duality solutions for the one-species aggregation equation. 
Such solutions are constructed by approximating the problem with particles, i.e. looking for a solution given by a finite
sum of Dirac delta functions. Particles attract themselves through the interacting potential $K$, when two particles collide,
they stick to form a bigger particle.

In this work, we extend this approach to the two species case. To do so, we need to modify the strategy to the problem at hand.
Indeed, collisions between particles of different species are more complex: particles can move together or separately after collision. 
This synchronising or non-synchronising dynamics implies several difficulties for the treatment of the dynamics of particles.
In fact, particles of different species can not stick when they collide.
Then an approximate problem is constructed by considering the transport equation with the a regularized velocity.
Then measure valued solutions are constructed by using a stability result.

An important advantage of this approach is that it allows to prove convergence of finite volume schemes.
Numerical simulations of the aggregation equation for the one-species case, which corresponds to the particular case of \eqref{aggregation_eq} when setting $\rho_2=0$, have been investigated by several authors. 
In \cite{Carrillo1} the authors propose a finite volume method consistent with the gradient flow structure of the equation, but no convergence result has been obtained. 
In \cite{Bertozzi1}, a Lagrangian method is proposed (see also the recent work \cite{Freda}). 
For the dynamics after blow up, a finite volume scheme which converges to the theoretical solution is proposed in \cite{VaucheletJ,CarrilloVauch}.
In the two-species case, the behaviour is more complex since the interaction between the two species can occur and they may synchronise or not i.e. move together or separately depending on the parameters of the models and the masses of species. 
A numerical scheme  illustrating this interesting synchronising or non-synchronising dynamics is provided in Section 6. 
In addition, a theoretical result on the convergence of the numerical approximation obtained
with our numerical scheme towards the duality solution is given. 
Such complex interactions phenomena have been observed experimentally in \cite{Casimir}.

System \eqref{aggregation_eq} can be derived from a hyperbolic limit of a kinetic chemotaxis model. In the case of two-velocities and in one space dimension, the kinetic chemotaxis model is given by
\begin{equation}\label{eq_cin_1d}
 \left\{
\begin{aligned}
& \partial_{t} f_{\alpha}^{\varepsilon}  +v\,\partial_{x} f_{\alpha}^{\varepsilon} = \frac{1}{\varepsilon}\int_{V}\left(T_{\alpha}[S](v',v)f_{\alpha}^{\varepsilon}(v')-T_{\alpha}[S](v,v') f_{\alpha}^{\varepsilon}(v)\right)dv',\quad \alpha=1,2,\,\,v\in V=\{\pm 1\},\\
& -\partial_{xx} S^{\varepsilon}+S^{\varepsilon}=\theta_1 \left(f_{1}^{\varepsilon}(1)   + f_{1}^{\varepsilon}(-1)\right) +  \theta_2 \left(f_{2}^{\varepsilon}(1)   + f_{2}^{\varepsilon}(-1)\right),
\end{aligned}
\right.
\end{equation}
where $f_{\alpha}^{\varepsilon}(x,v,t)$ stands for the distribution function of $\alpha$-th species at time $t$, position $x$ and velocity $v$, $S^{\varepsilon}(t,x)$ is the concentration of the chemical substance,
$T_{\alpha}[S](v,v')$ is the tumbling kernel from direction $v \in V$ to direction $v' \in V$ and $\varepsilon>0$ is a small parameter.
This tumbling kernel being affected by the gradient of the chemoattractant, is chosen as in \cite{DolakSch}
\begin{equation}\label{Turning}
 T_{\alpha}[S](v,v')=\psi_{\alpha}\left(1+\chi_{\alpha} v\partial_x S\right),
\end{equation}
where $\psi_{\alpha}$ is a positive constant called  the natural tumbling kernel and $\chi_{\alpha}$ is the chemosensitivity to the chemical $S$.
This kinetic model for chemotaxis has been introduced in \cite{OthmerAlt} to model the run-and-tumble process.
Existence of solutions to this two species kinetic system has been studied in \cite{Almeida}.

Summing and substracting equations \eqref{eq_cin_1d} with respect to $v=\pm 1$ for $f_\alpha^{\varepsilon}$ yields
\begin{equation}\label{first_order}
 \partial_{t} \rho_{\alpha}^{\varepsilon} + \partial_{x}J_{\alpha}^{\varepsilon}=0,
\end{equation}
\begin{equation}\label{second_order}
 \partial_{t} J_{\alpha}^{\varepsilon} + \partial_{x}\rho_{\alpha}^{\varepsilon}=\frac{2\psi_\alpha}{\varepsilon} (\chi_{\alpha}\partial_{x}S^{\varepsilon}\rho_{\alpha}^{\varepsilon}-J_{\alpha}^{\varepsilon}),\quad \alpha=1,2,
\end{equation}
where $\rho_{\alpha}^{\varepsilon}:=f_\alpha ^{\varepsilon}(1)+f_\alpha ^{\varepsilon}(-1)$ and $J_{\alpha}^{\varepsilon}:=(f_\alpha ^{\varepsilon}(1)-f_\alpha ^{\varepsilon}(-1))$.
Taking formally the limit $\eps\to 0$ in \eqref{second_order}, 
we deduce that $J_{\alpha}^{\varepsilon} \rightharpoonup \chi_{\alpha}\partial_x S^0 \rho_{\alpha}^0$ in the sense of distributions.
Injecting in \eqref{first_order}, we deduce formally that $\rho_\alpha^0$ satisfies the limiting equation:
\begin{equation}\label{limit_eq}
 \partial_{t}\rho_{\alpha}^0 + \chi_{\alpha} \partial_{x}((\partial_{x}S^{0}) \rho_{\alpha}^0 )=0,
\end{equation}
where $S^0$ satisfies the elliptic equation:
\begin{equation*}
-\partial_{xx} S^0 +S^0=\theta_1 \rho_1^0+\theta_2 \rho_2^0.
\end{equation*}
This latter equation can be solved explicitly on $\R$ and $S^0$ is given by
\begin{equation}\label{Kchemo}
 S^0=K*(\theta_1 \rho_1^0+\theta_2 \rho_2^0),\quad K=\frac{1}{2} e^{-\abs{x}}.
\end{equation}
Then we recover system \eqref{aggregation_eq}. This formal computation can be made rigorous.
The rigorous derivation of \eqref{limit_eq} from system \eqref{eq_cin_1d} will be proved in this work.\\

The paper is organized as follows. We first recall some basic notations and notions about the duality solutions and state our main results. 
Section 3 is devoted to the derivation of the macroscopic velocity used to define properly the product $a(\rho)\rho_{\alpha}$ and duality solutions. 
Existence and uniqueness of duality solutions are proved in Section 4, as well as its equivalence to gradient flow solutions. 
The convergence of the kinetic model \eqref{eq_cin_1d} as $\eps\to 0$ towards the aggregation model \eqref{limit_eq}-\eqref{Kchemo} is shown in Section 5.
Finally, a numerical scheme that captures the synchronising and non-synchronising behaviour of the aggregate equation is studied in Section 6, as well as several numerical 
examples showing the synchronising and non-synchronising dynamics.

\section{Notations and main results}
\subsection{Notations}
We will make use of the following notations. Let $T>0$, we denote
\begin{itemize}
 \item $L^1_+(\R)$ is the space of nonnegative functions of $L^1(\R)$.
 \item $C_0(\R)$ is the space of continuous functions that vanish at infinity.
 \item $\mathcal{M}_{\text{loc}}(\R)$ is the set of local Borel measures, $\mathcal{M}_{\text{b}}(\R)$ those whose total variation is finite:
 \begin{equation*}
  \mathcal{M}_b(\R)=\left\{\mu \in \mathcal{M}_{\text{loc}}(\R),\abs{\mu}(\R)<+\infty \right\}.
 \end{equation*}
\item $\mathcal{S}_\mathcal{M}=C([0,T],\mathcal{M}_{\text{b}}(\R)-\sigma(\mathcal{M}_{\text{b}}(\R),C_0(\R)))$ is the space of time-continuous bounded Borel measures endowed with the weak topology.
\item $\mathcal{P}_2(\R)$ is the Wasserstein space of order 2:
\begin{equation*}
 \mathcal{P}_2(\R)=\left\{\mu \quad \text{nonnegative borel measures in }\R \quad \text{s.t } \abs{\mu}(\R)=1, \int_{\R} \abs{x}^2 \mu(dx)< \infty \right\}.
\end{equation*}
\item For $H \in C(\R \backslash \{0\})$, we define $\widehat{H}$:
\begin{equation*}
 \widehat{H}=
 \left\{
 \begin{aligned}
  & H(x),\quad\text{for } x \neq 0,\\
  & 0, \quad \text{else}.
 \end{aligned}
 \right.
\end{equation*}
\end{itemize}
We notice that if $K$ satisfies \ref{h2} and \ref{h4}, we have by taking $y=-x$ in \ref{h4}
and using \ref{h2} that, $\forall\,x\in \RR$, $\partial_xK(x) x \leq \lambda x^2$.
We deduce that \ref{h4} holds for $\widehat{\pa_xK}$ i.e.:
\begin{equation}\label{Kchapobound}
\forall\, x,y \in \R, \quad (\widehat{\partial_x K}(x) -\widehat{\partial_x K}(y))(x-y) \leq \lambda (x-y)^2.
\end{equation}

We recall a compactness result on $\mathcal{M}_{\text{b}}(\R)-\sigma(\mathcal{M}_{\text{b}}(\R),C_0(\R))$. If there exists a sequence of bounded measures $\mu^n \in \mathcal{M}_{\text{b}}(\R)$
such that their total variations $\abs{\mu^n}(\R)$ are uniformly bounded, then there exists a subsequence of $\mu^n$ that converges weakly to $\mu$ in $\mathcal{M}_{\text{b}}(\R)$.

\subsection{Duality solutions}
For the sake of completeness, we recall the notion of duality solutions 
which has been introduced in \cite{Bouchut} for one dimensional linear scalar conservation law 
with discontinuous coefficients.
Let us then consider the linear conservation equation:
\begin{equation}\label{duality_eq}
\partial_t \rho+\partial_x (b(t,x)\rho)=0, \quad \text{in }  ]0,T[ \times \R,
 \end{equation}
with $T>0$. We assume weak regularity of the velocity field $b \in L^{\infty}(]0,T[\times \R)$
and $b$ satisfies the so-called one-sided Lipschitz (OSL) condition:
\begin{equation}\label{OSL}
\partial_x b\leq \gamma(t),\gamma \in L^1(]0,T[),\quad \text{in the sense of distributions}.
\end{equation}
In order to define duality solutions, we introduce the related backward problem
\begin{equation}\label{backward_eq}
\left\{
\begin{aligned}
 & \partial_t p + b \partial_x p=0,\\
 & p(T,\cdot)=p^T \in Lip_{loc}(\R).
\end{aligned}
\right.
\end{equation}
We define the set of exceptional solutions $\mathcal{E}$ as follows 
\begin{equation*}
 \mathcal{E}:=\left\{ p \in Lip_{loc}(]0,T[\times \R) \quad \text{solution to \eqref{backward_eq} with } p^T=0\right\}.
\end{equation*}
\begin{definition}[\textit{Reversible solutions} to \eqref{backward_eq}]\
We say that $p$ is a \textit{reversible} solution to \eqref{backward_eq} if and only if $p \in Lip_{loc}(]0,T[ \times \R)$ satisfies \eqref{backward_eq}
and is locally constant on $\mathcal{V}_e$, where $\mathcal{V}_e$ is defined by 
\begin{equation*}
 \mathcal{V}_e:=\left\{ (t,x) \in ]0,T[ \times \R; \; \exists p_e \in \mathcal{E} \, p_e(t,x)\neq 0\right\}.
\end{equation*}
\end{definition}
\begin{definition}[\textit{Duality solutions} to \eqref{duality_eq}, see \cite{Bouchut}]
We say $\rho \in \mathcal{S}_\mathcal{M}$ is a \textit{duality solution} to \eqref{duality_eq} in $]0,T[$ if for any $0<\tau\leq T$, and any $p$ reversible solution to \eqref{backward_eq} compactly supported in $x$,
the function $\displaystyle t\rightarrow \int_{\R}  p(t,x) \rho(t,dx)$ is constant on $[0,\tau]$.
\end{definition}
The following result shows existence and weak stability for duality solutions provided that  the velocity field satisfied the one-sided-Lipschitz condition.
\begin{theorem}[Theorem 2.1 in \cite{Bouchut}]\label{JBouchut}\ 
 \begin{enumerate}
  \item Given $\rho^{ini}\in \mathcal{M}_{b}(\R)$. Under the assumption \eqref{OSL}, there exists a unique $\rho \in \mathcal{S}_\mathcal{M}$, duality solution to \eqref{duality_eq},  such that $\rho(0,\cdot)=\rho^{ini}$.
  \item There exists a bounded Borel function $\hat{b}$, called universal representative of $b$ such that $\hat{b}=b$ a. e., and for any duality solution $\rho$,
  \begin{equation*}
   \partial_t \rho+\partial_x (\hat{b}\rho)=0,\quad \text{in the distributional sense.}
  \end{equation*}
 \item  Let $(b_n)_{n\in\N}$ be a bounded sequence in $L^{\infty}(]0,T[\times \R)$, with $b_n \rightharpoonup  b$ in $L^{\infty}(]0,T[ \times \R)-w*$. Assume that $\partial_x b_n\leq \gamma^n (t)$, where
 $(\gamma^n)_{n\in \N}$ is bounded in $L^1(]0,T[)$. Consider a sequence $\rho_n \in \mathcal{S}_\mathcal{M}$ of duality solutions to 
 \begin{equation*}
  \partial_t \rho_n +\partial_x(b_n \rho_n)=0,\quad \text{in } ]0,T[\times \R,
 \end{equation*}
such that $\rho_n(0,\cdot)$ is bounded in $\mathcal{M}_b(\R)$, and $\rho_n(0,\cdot)\rightharpoonup  \rho^{ini} \in \mathcal{M}_b(\R)$.
Then $\rho_n \rightharpoonup \rho$ in $\mathcal{S}_\mathcal{M}$, where $\rho \in \mathcal{S}_\mathcal{M}$ is the duality solution to
\begin{equation*}
 \partial_t \rho+\partial_x (b \rho)=0, \quad \text{in } ]0,T[ \times \R ,\quad \rho(0,\cdot)=\rho^{ini}.
\end{equation*}
Moreover, $\hat{b}_n \rho_n \rightharpoonup \hat{b} \rho$ weakly in $\mathcal{M}_b(]0,T[\times \R)$.
 \end{enumerate}
\end{theorem}
\subsection{Main results}
Up to a rescaling, we can assume without loss of generality that the total mass of each species
is normalized to 1. Then we will work in the space of probabilities for densities.

The first theorem states the existence and uniqueness of duality solutions for system \eqref{aggregation_eq} and its equivalence with the gradient flow solution considered in \cite{Francesco}.
\begin{definition}(Duality solutions for system \eqref{aggregation_eq})\label{defdual}
 We say that $(\rho_1,\rho_2)\in C([0,T],\mathcal{M}_b(\R)^2)$ is a duality solution to \eqref{aggregation_eq} if there exists $\hat{a}(\rho) \in L^{\infty}((0,T) \times \R)$ and $\gamma \in L^1_{loc}([0,T])$ satisfying
 $\partial_x \hat{a} \leq \gamma$ in the sense of distributions, such that for all $0<t_1<t_2<T$,
\begin{equation}\label{distributional_eq}
\partial_t \rho_{\alpha}+\chi_{\alpha} \partial_x (\hat{a}(\rho)\rho_{\alpha})=0, \quad \text{for } \alpha=1,2, \quad \rho=\theta_1 \rho_1+\theta_2 \rho_2,
\end{equation}
in the sense of duality on $(t_1,t_2)$ and $\hat{a}(\rho)=\partial_x K*\rho$ a.e.
We emphasize that the final datum for \eqref{distributional_eq} should be $t_2$ instead of $T$.
\end{definition}

Then, we have the following existence and uniqueness result:
\begin{theorem}[Existence, uniqueness of duality solution and equivalence to gradient flow solution]\ \label{Existence}
 Let $T>0$ and $(\rho_{1}^{ini},\rho_{2}^{ini})\in \mathcal{P}_2(\R)^2$. Under assumptions \ref{h1}--\ref{h4}, there exists a unique duality solution $(\rho_1,\rho_2) \in C([0,T],\mathcal{P}_2(\R)^2)$ to \eqref{aggregation_eq}  in the sense of Definition \ref{defdual}
with $(\rho_1,\rho_2)(t=0)=(\rho_{1}^{ini},\rho_{2}^{ini})$ such that 
\begin{equation}\label{macro_velocity}
\hat{a}(\rho):=\int_{\R}\widehat{\partial_x K}(x-y)\rho(t,dy),\quad \rho=\theta_1 \rho_1+\theta_2 \rho_2.
\end{equation}
\item This duality solution is equivalent to the gradient flow solution defined in \cite{Francesco}.
\end{theorem}
In our second main result, we prove the convergence of the kinetic model \eqref{eq_cin_1d} towards the aggregation model.
\begin{theorem}[Hydrodynamical limit of the kinetic model]\label{theo1.3}
Assume that $\chi_{\alpha}(\theta_1+\theta_2)<1$ for $\alpha=1,2$. Let $T>0$ and $(f_{\alpha}^{\varepsilon},S^{\varepsilon})$ be a solution to the kinetic-elliptic equation  \eqref{eq_cin_1d} 
such that $f_{\alpha}^{\varepsilon}(t=0)=f_{\alpha}^{ini}$ and $f_{\alpha}^{ini} \in L^{\infty}\cap L^1_{+}(\R)$ and $\displaystyle \int_{\R} x^2 f_{\alpha}^{ini}dx<\infty$.\\
Then, as $\varepsilon\rightarrow 0$, $(f_{\alpha}^{\varepsilon},S^{\varepsilon})$ converges in the following sense:
\begin{equation*}
 \begin{aligned}
  & \rho_{\alpha}^{\varepsilon} := f_{\alpha}^{\varepsilon}(1) + f_{\alpha}^{\varepsilon}(-1) \rightharpoonup \rho_{\alpha} \quad \text{weakly in }\quad \mathcal{S}_{M},\quad \text{for }\alpha=1, 2,  \\
  & S^{\varepsilon}\rightharpoonup S  \quad \text{in}\quad C([0,T], W^{1,\infty}(\mathbb{R}))-weak,
 \end{aligned}
\end{equation*}
where $\rho_{\alpha}$ is the unique duality solution of \eqref{limit_eq} and $S=K*(\theta_1 \rho_1+\theta_2 \rho_2)$ given in Theorem~\ref{Existence}.
\end{theorem}
The condition $\chi_{\alpha}(\theta_1+\theta_2)<1$ in the previous theorem is needed to guarantee that the tumbling kernel $T_\alpha[S]$ defined in \eqref{Turning} is positive.

To conclude this Section on the main results, we emphasize that, a finite volume scheme to simulate \eqref{distributional_eq} is proposed in Section 6 and its convergence towards duality solutions is stated in Theorem~\ref{convergence_numerical_scheme}.

\section{Macroscopic velocity}
In this section, we find the representative $\hat{a}$ of $a$ for which existence and uniqueness of duality solutions hold. 
To this end, we consider the similar system of transport equations to \eqref{aggregation_eq} associated to the velocity $a^n$ which converges to $a$.
Next, the limit of the product $a^n(\rho_{\alpha}^n) \rho_{\beta}^n$ for $\alpha,\beta=1,2$ is computed.

\subsection{Regularisation}
We build a sequence $(a^n)_{n\in \N}$ which converges to $a$ by considering the sequence of regularised kernels $\partial_x K^{n}$ approaching $\partial_x K$.
\begin{lemma} \label{uniform_bound}
 Let $(\partial_x K^{n})_{n\in \N}$ be the sequence of regular kernels defined by 
 \begin{equation*}
 \partial_{x} K^{n}(x)=
 \left\{
 \begin{aligned}
  & \partial_x K(x),\quad \text{for } \abs{x}>\frac{1}{n},\\
  & n\partial_x K\left(\frac{1}{n}\right)x,\quad \text{else}.
 \end{aligned}
\right.
\end{equation*}
Then
\begin{equation*}
 \partial_x K^n \in C^0(\R),\quad \forall x\in \R,\quad \partial_x K^n(-x)=-\partial_x K^n(x),
\end{equation*}
and 
\begin{equation*}
\|\partial_x K^n\|_\infty \leq \|\partial_x K\|_\infty,\quad \partial_{xx} K^n \leq \lambda \quad \text{in the distributional sense}.
\end{equation*}
\end{lemma}
\begin{proof}
From \ref{h1}, $\partial_x K\in C^0(\R \backslash \{0\})$ and since $\partial_x K^n$ is continuous at $\pm \frac{1}{n}$, we conclude that $\partial_x K^n \in C^0(\R)$.
From \ref{h2}, we deduce that $\partial_x K$ is an odd function. Using the definition of $\partial_x K^n$ and \ref{h3}, we get that $\|\partial_x K^n\|_\infty \leq \|\partial_x K\|_\infty$.
From the construction of $\partial_x K^n$, we have that $\partial_x K^n=\partial_x K$ outside the interval $[-\frac{1}{n},\frac{1}{n}]$ and from \ref{h4} one sees 
$\partial_{xx} K^n \leq \lambda$ in  $\R \backslash (-\frac{1}{n},\frac{1}{n})$ in the sense of distributions. In addition, if we take $x=-\frac{1}{n}$ and $y=\frac{1}{n}$ in \ref{h4}, we have that
\begin{equation*}
 n\partial_x K\left(\frac{1}{n}\right)\leq \lambda.
\end{equation*}
Since $\partial_{xx} K^n=n\partial_x K\left(\frac{1}{n}\right)$ in $[-\frac{1}{n},\frac{1}{n}]$, 
we conclude that $\partial_{xx} K^n \leq \lambda$ in $[-\frac{1}{n},\frac{1}{n}]$ in the sense of distributions.
Finally, we obtain that $\partial_{xx} K^n \leq \lambda$ in the sense of distributions.
\end{proof}
In the rest of the paper, the notation $\partial_x K^n$ will refer to the 
regularised kernels of Lemma \ref{uniform_bound}.
Given $\partial_x K^n$, the velocity $a^n$ is defined similarly to \eqref{macro_velocity} as
\begin{equation}\label{f11-1}
 \forall \rho \in \mathcal{S}_\mathcal{M},\quad a^n(\rho):=\int_{\R}\partial_x K^n(x-y)\rho(t,dy).
\end{equation}
In the following lemma, we show that if $\rho_{\alpha}^n$ and $\rho_{\beta}^n$ admit weak limits $\rho_{\alpha}$ and $\rho_{\beta}$ respectively in $\mathcal{S}_\mathcal{M}$, then 
the limit of the product $a^n(\rho_{\alpha}^n)\rho_{\beta}^n$ is $\hat{a}(\rho_{\alpha})\rho_{\beta}$. 
Contrary to \cite{Poupaud2} where the two-dimensional case is considered, 
this limiting measure does not charge the diagonal.
\begin{lemma}\label{technical_lemma}
For $\alpha=1,2$, let $\{\rho_{\alpha}^{n}\} \in \mathcal{S}_\mathcal{M}$ be a sequence such that $\forall n\in \N,\forall t \in [0,T], \abs{\rho_{\alpha}^{n}}(t,\R)=M_{\alpha}$.
Suppose that there exists $\rho_{\alpha}$ in $\mathcal{S}_\mathcal{M}$ such that 
 \begin{equation*}
  \rho_{\alpha}^{n} \rightharpoonup \rho_{\alpha} \quad \text{weakly in }  \mathcal{S}_\mathcal{M},
 \end{equation*}
 Then, we have
 \begin{equation*}
  a^n(\rho_{\alpha}^n)\rho_{\beta}^n \rightharpoonup \hat{a}(\rho_{\alpha})\rho_{\beta} \quad \text{weakly in }  \mathcal{M}_b([0,T] \times \mathbb{R}),\quad \alpha,\beta=1,2,
 \end{equation*}
where  $a^n(\cdot)$ and $\hat{a}(\cdot)$ are defined in \eqref{macro_velocity},\eqref{f11-1} respectively.
That is for $\phi \in C_0([0,T] \times \R)$, 
\begin{equation*}
 \int_0^T\int_{\mathbb{R}^2} \partial_x K^{n}(x-y)\rho_{\alpha}^{n}(t,dx) \rho_\beta ^{n}(t,dy) \phi(t,x)dt\rightarrow \int_0^T \int_{\R^2} \widehat{\partial_x K}(x-y)\rho_{\alpha}(t,dx) \rho_\beta (t,dy) \phi(t,x)dt .
 \end{equation*}
\end{lemma}
\begin{proof}
Before starting the proof of the lemma, we first introduce some notations which simplify the computations 
\begin{equation}\label{def}
\begin{aligned}
  \mu^n(t,dx,dy)&:=\rho_{\alpha}^n(t,dx) \otimes \rho_{\beta}^n(t,dy),\quad E_n:=\left\{(x,y)\in \R^2, x\neq y, \abs{x-y}\leq \frac{1}{n}\right\},\\
  \mu(t,dx,dy)&:=\rho_{\alpha}(t,dx) \otimes \rho_{\beta}(t,dy).\\
\end{aligned}
\end{equation}
For $\phi \in C_0([0,T] \times \R)$, we denote
\begin{equation*}
 A_n(t):=\int_{\mathbb{R}^2} \partial_x K^{n}(x-y)\mu^{n}(t,dx,dy) \phi(t,x)-\int_{\R^2} \widehat{\partial_x K} (x-y)\mu(t,dx,dy) \phi(t,x),
\end{equation*}
\textbf{Step 1:} Convergence almost everywhere in time of $A_n(t)$.\vspace{0.5cm}\\
Since $\partial_x K^{n}(0)=0$, we have 
\begin{equation*}
\begin{aligned}
A_n(t)&=\int_{\R^2} \widehat{\partial_x K^{n}}(x-y)\mu^{n}(t,dx,dy) \phi(t,x)-\int_{\R^2} \widehat{\partial_x K}(x-y)\mu(t,dx,dy)\phi(t,x),\\
\phantom{A}&=\text{I}_n(t)+\text{II}_n(t),
\end{aligned}
\end{equation*}
where $\text{I}_n(t)$ and $\text{II}_n(t)$ are defined by 
\begin{equation*}
\begin{aligned}
 \text{I}_n(t)&:=\int_{\R^2}\left(\widehat{\partial_x K^{n}}(x-y)-\widehat{\partial_x K}(x-y)\right)\mu^{n}(t,dx,dy) \phi(t,x),\\
  \text{II}_n(t)&:=\int_{\R^2}\widehat{\partial_x K}(x-y)\left(\mu^{n}(t,dx,dy)-\mu(t,dx,dy)\right)\phi(t,x).
\end{aligned}
\end{equation*}
From the definition of $\partial_x K^n$ in Lemma~\ref{uniform_bound}, it follows that
\begin{equation*}
 \text{I}_n(t)=\int_{E_n}\left(\partial_x K^{n}(x-y)-\partial_x K(x-y)\right)\mu^{n}(t,dx,dy) \phi(t,x).
\end{equation*}
The estimate on $\|\partial_x K^n\|_{L^{\infty}}$ in Lemma~\ref{uniform_bound} and \ref{h3} imply that 
\begin {equation*}
 \abs{\text{I}_n(t)}\leq 2\|\phi\|_{L^{\infty}}\|\partial_x K\|_{L^{\infty}} \mu^n(t,E_n),
\end {equation*}
with $\mu^n$ and $E_n$ defined in \eqref{def}.\\
Let $\varepsilon>0$. Since the set $E_n$ converges to the empty set, there exists $N \in \mathbb{N}$ such that  $\forall n \geq N$,
\begin{equation}\label{ineq11}
\mu(t,E_n) \leq \varepsilon.
\end{equation}
For all $n\geq N$, we observe that $E_n \subset E_N$, we have
\begin{equation}\label{ineq}
\mu^n(t,E_n)\leq \mu^n(t,E_N) \leq (\mu^n-\mu)(t,E_N)+\mu(t,E_N). 
\end{equation}
From the weak convergence of $\rho_\alpha^n$, $\alpha=1,2$, we note that the sequence $\mu^n(t,\cdot)$ 
converges weakly to $\mu(t,\cdot)$. 
Since the total variation of $\mu^n(t,\cdot)$ is constant in $n$, the tight convergence is achieved. 
Then, there exists $N'$ such that $\forall n \geq N'\geq N$
\begin{equation*}
\abs{\mu^n-\mu}(t,E_N) \leq \varepsilon.
\end{equation*}
From \eqref{ineq} and \eqref{ineq11}, we conclude that $\forall n \geq N'\geq N$,
\begin{equation*}
\mu^n(t,E_n) \leq 2 \varepsilon.
\end{equation*}
Hence, for all $n \geq N'$, we get
\begin{equation}\label{f11-5}
 \abs{\text{I}_n(t)} \leq 2\|\phi\|_{L^{\infty}}\|\partial_x K\|_{L^{\infty}} \mu^n(t,E_n) \leq 4\|\phi\|_{L^{\infty}}\|\partial_x K\|_{L^{\infty}} \ \varepsilon.
\end{equation}
We deduce that $\text{I}_n(t) \longrightarrow 0$.

Next, we show that $\text {II}_n(t)$ tends to zero.
\begin{equation*}
 \begin{aligned}
  \text {II}_n(t)&=\int_{\R^2}(\widehat{\partial_x K}(x-y)-\widehat{\partial_x K^R}(x-y))\phi(t,x)(\mu^n(t,dx,dy)-\mu^n(t,dx,dy))\\
           &+\int_{\R^2}\widehat{\partial_x K^R}(x-y)\phi(t,x)(\mu^n(t,dx,dy)-\mu(t,dx,dy)),\\
           &:=\text {II}_n^1(t)+\text {II}_n^2(t),
 \end{aligned}
\end{equation*}
where $R$ is an integer which will be fixed later.
From the construction of $\partial_x K^R$ in Lemma~\ref{uniform_bound}, we get
\begin{equation*}
 \text {II}_n^1=\int_{E_R}(\partial_x K(x-y)-\partial_x K^R(x-y))\phi(t,x)(\mu^n(t,dx,dy)-\mu(t,dx,dy)).
\end{equation*}
Therefore, one has
\begin{equation*}
\abs{\text {II}_n^1(t)}\leq 2\|\partial_x K\|_{L^{\infty}}\|\phi\|_{L^{\infty}}\left(\mu^n(t,E_R)+\mu(t,E_R)\right).
\end{equation*}
Let $\varepsilon>0$.
Using \eqref{ineq}, by the same token as previously, there exists $N$ such that for all $n\geq N$,
\begin{equation*}
 \mu^n(t,E_N) \leq 2 \varepsilon,\quad \mu(t,E_n)\leq \varepsilon, 
\end{equation*}
Setting $R=N$, we conclude that for all $n \geq N$,
\begin{equation*}
 \abs{\text {II}_n^1(t)} \leq 6 \varepsilon \|\partial_x K\|_{L^{\infty}}\|\phi\|_{L^{\infty}}.
\end{equation*}
For $\text {II}_n^2(t)$, we notice that $\partial_x K^N(x-y)\phi(t,x)$ is a continuous function that vanishes on the diagonal $(x,x)$ and we have
\begin{equation*}
 \int_{\R^2}\widehat{\partial_x K^N}(x-y)\phi(t,x)(\mu^n-\mu)(t,dx,dy)= \int_{\mathbb{R}^2}\partial_x K^N(x-y)\phi(t,x)(\mu^n-\mu)(t,dx,dy).
\end{equation*}
The tight convergence of $\mu^n$ to $\mu$ implies that there exists $N''>0$ such that for all $n\geq N''$
\begin{equation*}
 \abs{\int_{\R^2}\widehat{\partial_x K^M}(x-y) \phi(t,x)(\mu^n-\mu)(t,dx,dy)} \leq \varepsilon.
\end{equation*}
Therefore for all $n\geq \max\{N',N''\}$, one has
\begin{equation}\label{f11-6}
 \abs{\text{II}_n(t)} \leq \varepsilon(1+6\|\partial_x K\|_{L^{\infty}}\|\phi\|_{L^{\infty}}).
\end{equation}
This implies that $\text {II}_n(t)$ converges to 0.\\
Combining \eqref{f11-5} and \eqref{f11-6}, we deduce that for almost every $t \in [0,T]$, $A_n$ converges to 0.\vspace{0.5cm}\\
\textbf{Step 2:} Lebesgue's dominated convergence theorem\vspace{0.5cm}\\
For all $t\in [0,T]$, we have that 
\begin{equation*}
 \abs{A_n(t)}\leq 2\|\phi\|_{L^{\infty}}\|\partial_x K\|_{L^{\infty}}M_{\alpha} M_{\beta}.
\end{equation*}
Since $A_n$ converges almost everywhere to 0, $\int_{0}^T A_n(t)dt$ converges to zero from Lebesgue's dominated convergence theorem.
\end{proof}

\subsection{OSL condition on the macrosocopic velocity}

\begin{proposition}\label{OSL_condition}
 Let $T,M$ be positive constants and $\rho \in \mathcal{S}_\mathcal{M}$ be a positive measure such that $\forall t\in [0,T], \abs{\rho}(t,\R)=M$. 
 Let $K$ be such that assumption \ref{h4} hold. Let $\hat{a}(\rho)$ and $a^n(\rho)$ be defined in \eqref{macro_velocity} and \eqref{f11-1} respectively.
 Then, there exists $\kappa \in L^1([0,T])$ such that 
 \begin{equation*}
 \partial_x \hat{a}(t,x) \leq \kappa(t),\quad \partial_x a^n(t,x) \leq \kappa(t),
\quad \mbox{ in the sense of distributions.}
 \end{equation*}
\end{proposition}
\begin{proof}
For $x,y \in \R$, we compute:
\begin{multline*}
  (\hat{a}(\rho)(t,x)-\hat{a}(\rho)(t,y))(x-y)=
  \int_{\R}(\widehat{\partial_x K}(x-z)-\widehat{\partial_x K}(y-z))(x-y)\rho(t,dz).
\end{multline*}
Using the $\lambda$-concavity of $K$, we deduce from \eqref{Kchapobound}
\begin{equation*}
 (\hat{a}(\rho)(t,x)-\hat{a}(\rho)(t,y))(x-y) \leq \lambda (x-y)^2 \int_{\R}\rho(t,dz) \leq \lambda M (x-y)^2.
\end{equation*}
Since $K^n$ is also $\lambda$ concave from the proof of Lemma \ref{uniform_bound}, we get the one-sided Lipschitz
estimate on $a^n$ by the same token as for $a$.
\end{proof}

\section{Existence and uniqueness of duality solutions}
\subsection{Proof of the existence of duality solutions in Theorem~\ref{Existence}}

The proof is divided into several steps. First, we construct an approximate problem for which the existence of duality solutions holds.
Then, we pass to the limit in the approximate problem to get the existence of duality solutions thanks to the weak stability of duality solutions stated in Theorem~\ref{JBouchut}
and recover Equation~\eqref{distributional_eq} from Lemma~\ref{technical_lemma}.
Finally, we recover the bound on the second order moment.\\

\textbf{Step 1}: Existence of duality solutions for the approximate problem\vspace{0.3cm}\\
The macroscopic velocity $a$ is replaced by an approximation $a^n$ defined in \eqref{f11-1} and the following system is considered:
\begin{equation}\label{macro_eq_eps}
  \partial_{t} \rho_{\alpha}^{n} + \chi_{\alpha} \partial_{x} \left(a^{n}(\theta_1\rho_{1}^{n}+\theta_2\rho_{2}^{n})\rho_{\alpha}^{n}\right) = 0,\quad \text{for }\alpha=1, 2.
\end{equation}
Since $\partial_x K^n$ is not Lipschitz continuous, we first consider $\partial_x K^{n,m}$ an approximation of $\partial_x K^n$ obtained by a convolution with a molifier.
The solution $\rho_{\alpha}^{n,m}$ to the following equation is investigated.
\begin{equation}\label{macro_eq_eps_eps}
 \partial_t \rho_{\alpha}^{n,m}+\chi_{\alpha} \partial_x \left(a^{n,m}(\theta_1 \rho_{1}^{n,m}+\theta_2 \rho_2^{n,m})\rho_{\alpha}^{n,m}\right)=0,\quad \text{for }\alpha=1,2,
\end{equation}
where $a^{n,m}$ is given by
\begin{equation*}
 a^{n,m}(\rho):=\int_\R \partial_x K^{n,m}(x-y) \rho(t,dy).
\end{equation*}
Applying Theorem 1.1 in \cite{Mercier} gives the existence of solutions $\rho_{\alpha}^{n,m}$ in $L^{\infty}([0,T],\mathcal{M}_b(\R))$ and $\abs{\rho_{\alpha}^{n,m}}(t,\R)=\abs{\rho_{\alpha}^{ini}}(\R)=1$. 
Since the velocity field $a^{n,m}$ is Lipschitz, $\rho^{n,m}_\alpha$ is a duality solution.
In addition, for $\phi \in C^{\infty}_c(\mathbb{R})$ we have for $\alpha=1,2$ the following estimate:
\begin{equation*}
 \frac{d}{dt}\left(\int_{\R}\rho_{\alpha}^{n,m} (t,dx)\phi(x)\right)
 =\int_{\R^2}\partial_x K^{n,m}(x-y)(\theta_1 \rho_1^{n,m}(t,dy)+\theta_2 \rho_2^{n,m}(t,dy))\rho_{\alpha}^{n,m}(t,dx) \partial_x \phi(x).
\end{equation*}
Then, 
\begin{equation}\label{estimate_rho}
 \abs{\frac{d}{dt}\left(\int_{\R}\rho_{\alpha}^{n,m}(t,dx)\phi(x)\right)} \leq \|\partial_x \phi\|_{L^{\infty}}\|\partial_x K\|_{L^{\infty}}(\theta_1+\theta_2).
\end{equation}
Using \eqref{estimate_rho} and the density of $C_c^{\infty}(\R)$ in $C_0(\R)$, we deduce that $\rho_{\alpha}^{n,m} \in \mathcal{S}_\mathcal{M}$.
Moreover, the equicontinuity of $\rho_{\alpha}^{n,m}$ in $\mathcal{S}_\mathcal{M}$ follows from \eqref{estimate_rho} and the density of $C_c^{\infty}(\R)$ in $C_0(\R)$.
Since $\abs{\rho_{\alpha}^{n,m}}(t,\R)=\abs{\rho_{\alpha}^{ini}}(\R)=1$, Ascoli Theorem gives the existence of a subsequence in $m$ of $\rho_{\alpha}^{n,m}$ which converges to a limit named $\rho_{\alpha}^n$ in $\mathcal{S}_\mathcal{M}$.
We pass to the limit when $m$ tends to infinity in Equation \eqref{macro_eq_eps_eps} and obtain that $\rho_{\alpha}^n$ satisfies \eqref{macro_eq_eps}.\\

\textbf{Step 2} : Extraction of a convergent subsequence of $\rho_{\alpha}^n$ and existence of duality solutions.\vspace{0.3cm}\\
As above, there exists a subsequence of $\rho_{\alpha}^n$ in $\mathcal{S}_\mathcal{M}$ such that
\begin{equation*}
 \rho_{\alpha}^n\rightharpoonup \rho_{\alpha} \quad \text{weakly in }\mathcal{S}_\mathcal{M},\quad \text{for }\alpha=1,2.
\end{equation*}
Let us find the equation satisfied by $\rho_{\alpha}$ in the distributional sense.
Let $\phi$ be in $C^{\infty}_c([0,T] \times \R)$. Since $\rho_{\alpha}^n$ satisfies \eqref{macro_eq_eps} in the distributional sense, we have 
\begin{equation*}
\int_0^T\int_{\R} \partial_t \phi (t,x)\rho_{\alpha}^n(t,dx)dt +\chi_{\alpha} \int_0^T \int_{\R}a^n(\theta_1\rho_1^n+\theta_2\rho_2^n)\rho_{\alpha}^n(t,dx)\phi(t,x)dt=0.
\end{equation*}
Using Lemma~\ref{technical_lemma}, we can pass to the limit in the latter equation and obtain, 
\begin{equation*}
\int_0^T\int_{\R} \partial_t \phi(t,x) \rho_{\alpha}(t,dx)dt+\chi_{\alpha} \int_0^T \int_{\R}\hat{a}(\theta_1\rho_1+\theta_2\rho_2)\rho_{\alpha}(t,dx)\phi(t,x)dt=0.
\end{equation*}
Thus $\rho_{\alpha}$ satisfies \eqref{distributional_eq} in the sense of distributions. 
From Proposition \ref{OSL_condition}, the macroscopic velocity $a^n(\rho^n)$ satisfies an uniform OSL condition. Then, by weak stability of duality solutions in (see Theorem~\ref{JBouchut} (3)), we deduce that 
\begin{equation*}
  \partial_t \rho_{\alpha}+\chi_{\alpha}\partial_x(\hat{a}(\rho)\rho_{\alpha})=0, \quad \text{for } \alpha=1,2,\quad \text{ in the sense of duality in } ]0,T[ \times \mathbb{R}.
\end{equation*}

\textbf{Step 3} : Finite second order moment.\vspace{0.3cm}\\
From Equation \eqref{distributional_eq}, we deduce that the first and second moments satisfy in the sense of distributions
\begin{equation*}
 \begin{aligned}
  \frac{d}{dt}\left(\int \abs{x}\rho_{\alpha}(t,dx)\right)&=-\int \sgn(x) \hat{a}(\rho) \rho_{\alpha}(t,dx),\\
  \frac{d}{dt}\left(\int \abs{x}^2 \rho_{\alpha}(t,dx)\right)&=-2\int \hat{a}(\rho) \rho_{\alpha}(t,dx),\quad \text{for }\alpha=1,2.
 \end{aligned}
\end{equation*}
Since $\rho_{\alpha}^{ini} \in \mathcal{P}_2(\R)$ and $\hat{a}(\rho)$ is bounded from \eqref{h2}, we deduce that the first two moments of $\rho_{\alpha}$ are finite, then $\rho_{\alpha}(t) \in \mathcal{P}_2(\R)$ for $t>0$.
\qed
\begin{remark}
 If we define the weighted center of mass of the system $x_c$ as follows:
 \begin{equation*}
  x_c(t):=\frac{\theta_1}{\chi_1} \int_\R x \rho_1(t,dx) + \frac{\theta_2}{\chi_2} \int_\R x \rho_2(t,dx).
 \end{equation*}
 We remark from straightforward computation that $\frac{d}{dt}x_c=0$. 
 Then the weighted center of mass is conserved for this system.
\end{remark}

\subsection{Proof of the uniqueness of duality solutions in Theorem~\ref{Existence}}
Uniqueness relies on a stability estimate in Wasserstein distance, which is the metric endowed in $\mathcal{P}_2(\R)$.
This Wasserstein distance $d_W$ is defined by (see e.g. \cite{Villani1,Villani2})
\begin{equation*}
 d_W(\nu,\mu)=\inf_{\gamma \in \Gamma(\nu,\mu)} \left\{\int_{\R^2} \abs{y-x}^2\gamma(dx,dy)\right\}^{1/2},
\end{equation*}
where $\Gamma(\mu,\nu)$ is the set of measures on $\R^2\times \R^2$ with marginals $\mu$ and $\nu$, i.e.,
\begin{multline*}
  \Gamma(\nu,\mu)=\left\{\gamma \in \mathcal{P}_2(\R \times \R), \forall \xi \in C_0(\R \times \R), \int_{\R^2} \xi(y_0) \gamma(dy_0,dy_1)=\int_{\R}\xi(y_0)\mu(dy_0),\right.\\
  \left.\int_{\R ^2} \xi(y_1) \gamma(dy_0,dy_1)=\int_{\R}\xi(y_1)\nu(dy_1)\right\}.
\end{multline*}
The Wasserstein distance $d_W$ takes a more pratical form in the one-dimensional setting. Indeed, in one space dimension, we have (see e.g \cite{Rachev,Villani1})
\begin{equation*}
 d_W(\nu,\mu)^2=\int_{0}^1 \abs{F_{\nu}^{-1}(z)-F_{\mu}^{-1}(z)}^2 dz, 
\end{equation*}
where $F_{\nu}^{-1}$ and $F_{\mu}^{-1}$ are the generalised inverse of cumulative distributions of $\nu$ and $\mu$, defined by
\begin{equation*}
 F_{\nu}^{-1}(z)=\text{inf}\Big\{x \in \R, \nu((-\infty,x))> z\Big\},\quad F_{\mu}^{-1}(z)=\text{inf}\Big\{x \in \R, \mu((-\infty,x))> z\Big\}.
\end{equation*}
This Wasserstein distance can be extended to the product space $\mathcal{P}_2(\R)\times  \mathcal{P}_2(\R)$. In the case at hand, we define $W_2(\nu,\mu)$ by
\begin{equation}\label{defW2}
 W_2(\nu,\mu)^2=\int_{0}^1 \abs{F_{\nu_1}^{-1}(z)-F_{\mu_1}^{-1}(z)}^2 dz +\frac{\chi_1 \theta_2}{\chi_2\theta_1}\int_{0}^1 \abs{F_{\nu_2}^{-1}(z)-F_{\mu_2}^{-1}(z)}^2 dz,
\end{equation}
where $\nu=\begin{pmatrix} \nu_1 \\ \nu_2 \end{pmatrix},\mu=\begin{pmatrix} \mu_1 \\ \mu_2 \end{pmatrix} \in \mathcal{P}_2(\R)\times \mathcal{P}_2(\R)$ and $F_{\nu_{\alpha}}^{-1}$, $F_{\mu_{\alpha}}^{-1}$ are 
the generalised inverse of cumulative distributions of  $\nu_{\alpha}$ and $\mu_{\alpha}$ for $\alpha=1,2$, respectively.
Using $W_2$ we prove a contraction inequality between duality solutions of \eqref{aggregation_eq}.
\begin{proposition}\label{uniq}
Let $\displaystyle \mu^{ini}=\begin{pmatrix} \mu_1^{ini} \\ \mu_2^{ini} \end{pmatrix}$ and $\displaystyle \nu^{ini}=\begin{pmatrix} \nu_1^{ini} \\ \nu_2^{ini}\end{pmatrix}$  be in $\mathcal{P}_2(\R)^2$.
We define $\displaystyle \mu=\begin{pmatrix} \mu_1 \\ \mu_2\end{pmatrix}$ and $\displaystyle \nu=\begin{pmatrix} \nu_1 \\ \nu_2\end{pmatrix}$ duality solutions of \eqref{aggregation_eq} with respectively the initial data $\mu^{ini},\nu^{ini}$.\\
Then $W_{2}(\mu,\nu)$ defined in \eqref{defW2} is bounded and satisfies the estimate:\\
\begin{equation*}
W_{2}(\mu,\nu) \leq W_{2}(\mu^{ini},\nu^{ini}) \exp{\left(2\lambda(\chi_1+\chi_2)(\theta_1+\theta_2) t\right)}.
\end{equation*}
\end{proposition}
\begin{proof}
Since $(\mu,\nu) \in {\mathcal P}_2(\R)^2 \times {\mathcal P}_2(\R)^2$, $W_2(\mu,\nu)$ is bounded. For the sake of clarity in the proof, we denote
\begin{equation*}
 F_{\alpha}^{-1}:=F_{\nu_{\alpha}}^{-1},\quad G_{\alpha}^{-1}:=F_{\mu_{\alpha}}^{-1},\quad \text{for }\alpha=1,2.
\end{equation*}
We also omit the argument $t$ in notations $F^{-1}_\alpha(t,x)$ and $G^{-1}_\alpha(t,x)$. Computing the derivative of $W_{2}(\mu,\nu)^2$ with respect to time,
\begin{equation*}
\begin{aligned}
  \partial_t W_{2}(\mu,\nu)^2&=2\int_0^1 \big(F_1^{-1}(x)-G_1^{-1}(x)\big)\big(\partial_t F_1^{-1}(x)-\partial_t G_1^{-1}(x)\big)dx\\
 & +2\frac{\chi_1 \theta_2}{\chi_2 \theta_1}\int_0^1 \big(F_2^{-1}(x)-G_2^{-1}(x)\big)\big(\partial_t F_2^{-1}(x)-\partial_t G_2^{-1}(x)\big) dx.
\end{aligned}
\end{equation*}
Straightforward and standard computations give that
\begin{equation*}
 \partial_t F_{\alpha}^{-1}(x)=\chi_{\alpha} \hat{a}(t,F_{\alpha}^{-1}(x)),\quad \partial_t G_{\alpha}^{-1}=\chi_{\alpha}\hat{a}(t,G_{\alpha}^{-1}(x)),\quad \text{for }\alpha=1,2.
\end{equation*}
From the definition of $\hat{a}$ in \eqref{macro_velocity}, we get
\begin{equation*}
 \partial_t F_1^{-1}(x)=\chi_1 \theta_1\int_0^1 \widehat{\partial_x K}(F_1^{-1}(x)-z)\mu_1(t,dz)+ \chi_1 \theta_2 \int_0^1 \widehat{\partial_x K}(F_1^{-1}(x)-z)\mu_2(t,dz).
\end{equation*}
Setting $z=F_1^{-1}(y)$ in the first integral and $z=F_2^{-1}(y)$ in the second one yields 
\begin{equation*}
 \partial_t F_1^{-1}(x)=\chi_1 \theta_1 \int_0^1 \widehat{\partial_x K}(F_1^{-1}(x)-F_1^{-1}(y))dy+ \chi_1\theta_2 \int_0^1 \widehat{\partial_x K}(F_1^{-1}(x)-F_2^{-1}(y))dy.
\end{equation*}
Similarly, we get
\begin{equation*}
 \partial_t G_1^{-1}(x)=\chi_1 \theta_1 \int_0^1 \widehat{\partial_x K}(G_1^{-1}(x)-G_1^{-1}(y))dy+ \chi_1\theta_2 \int_0^1 \widehat{\partial_x K}(G_1^{-1}(x)-G_2^{-1}(y))dy.
\end{equation*}
Using the oddness of $\partial_x K$, we can symmetrise the terms in the right-hand side of $\partial_t F_1^{-1}$, $\partial_t G_1^{-1}$. One gets
\begin{multline*}
  \int_0^1 \big(F_1^{-1}(x)-G_1^{-1}(x)\big)\big(\partial_t F_1^{-1}(x)-\partial_t G_1^{-1}(x)\big)dx=\\
   \frac{1}{2} \chi_1\theta_1 \int_0^1 \int_0^1 \left(\widehat{\partial_x K}(F_1^{-1}(x)-F_1^{-1}(y))-\widehat{\partial_x K}(G_1^{-1}(x)-G_1^{-1}(y))\right) \times\\
   \left(F_1^{-1}(x)-G_1^{-1}(x)-\big(F_1^{-1}(y)-G_1^{-1}(y) \big)\right)dx \, dy\\
  +\chi_1\theta_2 \int_0^1 \int_0^1\left(\widehat{\partial_x K}(F_1^{-1}(x)-F_2^{-1}(y))-\widehat{\partial_x K}(G_1^{-1}(x)-G_2^{-1}(y))\right)\big(F_1^{-1}(x)-G_1^{-1}(x)\big)dy\,dx.\\
\end{multline*}
Similar computations can be carried out for $\int_0^1 \big(F_2^{-1}(t,x)-G_2^{-1}(t,x)\big)\big(\partial_t F_2^{-1}(t,x)-\partial_t G_2^{-1}(t,x)\big)$. 
Finally, $\partial_t W_2(\nu,\mu)^2$ reads
\begin{equation}\label{derivative_W}
 \begin{aligned}
  \partial_t W_2(\nu,\mu)^2&=
 \chi_1 \theta_1 \int_0^1 \int_0^1 \left(\widehat{\partial_x K}(F_1^{-1}(x)-F_1^{-1}(y))-\widehat{\partial_x K}(G_1^{-1}(x)-G_1^{-1}(y))\right) \times\\
   &\left(F_1^{-1}(x)-G_1^{-1}(x)-\big(F_1^{-1}(y)-G_1^{-1}(y) \big)\right)dx \, dy\\
 &+\frac{\chi_1 \theta_2^2}{\theta_1} \int_0^1 \int_0^1 \left(\widehat{\partial_x K}(F_2^{-1}(x)-F_2^{-1}(y))-\widehat{\partial_x K}(G_2^{-1}(x)-G_2^{-1}(y))\right) \times\\
  & \left(F_2^{-1}(x)-G_2^{-1}(x)-\big(F_2^{-1}(y)-G_2^{-1}(y) \big)\right)dx \, dy\\
 &+ 2\chi_1 \theta_2 \int_0^1 \int_0^1 \left(\widehat{\partial_x K}(F_1^{-1}(x)-F_2^{-1}(y))-\widehat{\partial_x K}(G_1^{-1}(x)-G_2^{-1}(y))\right) \times\\
   &\left(F_1^{-1}(x)-F_2^{-1}(y)-\big(G_1^{-1}(x)-G_2^{-1}(y) \big)\right)dx \, dy.\\
 \end{aligned}
\end{equation}
Applying inequality \eqref{Kchapobound} to \eqref{derivative_W} and using Young's inequality yields
\begin{equation*}
  \partial_t W_2(\nu,\mu)^2 \leq 4 \chi_1 \lambda \times \left((\theta_1+\theta_2)\int_0^1 (F_1^{-1}(x)-G_1^{-1}(x))^2 dx+(\theta_2+\frac{\theta_2^2}{\theta_1})\int_0^1 (F_2^{-1}(x)-G_2^{-1}(x))^2 dx\right).
\end{equation*}
By definition of $W_2$ \eqref{defW2}, we conclude
\begin{equation*}
\partial_t W_2(\nu,\mu)^2\leq 4\lambda (\chi_1+\chi_2)(\theta_1+\theta_2)W_2(\nu,\mu)^2.
\end{equation*}
Then the result follows from Gronwall's Lemma.
\end{proof}

{\bf Proof of uniqueness.}
From Proposition~\ref{uniq}, it is clear that if $\mu^{ini}=\nu^{ini}$, then $\mu=\nu$. 
We deduce uniqueness of duality solution in Theorem~\ref{Existence}.

\subsection{Equivalence with gradient flow}
We recall that $\mu \in AC^2([0,T],\mathcal{P}_2(\R) \times \mathcal{P}_2(\R))$ if $\mu$ is locally H\"{o}lder continuous of exponent $1/2$ with respect to the Wasserstein distance $W_2$ in $\mathcal{P}_2(\R) \times \mathcal{P}_2(\R)$.
\begin{proposition}
 Let assumptions of Theorem~\ref{Existence} hold. 
 Given $\rho^{ini}=\begin{pmatrix}\rho_1^{ini} \\ \rho_2^{ini} \end{pmatrix} \in \mathcal{P}_2(\R) \times \mathcal{P}_2(\R) $. Let $\rho=\begin{pmatrix}\rho_1 \\ \rho_2 \end{pmatrix}$ and $\tilde{\rho}=\begin{pmatrix}\tilde{\rho}_1 \\ \tilde{\rho}_2 \end{pmatrix}$
 be respectively the duality and gradient flow solution. Then, we have $\rho \in AC^{2}([0,T],\mathcal{P}_2(\R)\times \mathcal{P}_2(\R))$ and $\rho=\tilde{\rho}$.
\end{proposition}
\begin{proof}
We have that $\begin{pmatrix}\hat{a}(\rho)\\\hat{a}(\rho)\end{pmatrix} \in L^1([0,T],L^2(\rho_1 \otimes \rho_2,\R^2))$.
This comes from the fact that $\partial_x K$ is bounded and 
 \begin{equation*}
 \abs{\int_\RR \widehat{\partial_x K} (x-y)(\theta_1 \rho_1(t,dy)+\theta_2 \rho_2(t,dy))} \leq \|\partial_x K\|_{L^{\infty}} (\theta_1+\theta_2).
 \end{equation*}
From Theorem 8.3.1 in \cite{Ambrosio}, we deduce that $\rho \in AC^{2}([0,T],\mathcal{P}_2(\R)\times \mathcal{P}_2(\R))$.
Since $\rho$ satisfies \eqref{distributional_eq} in the distributional sense, we deduce by uniqueness of such solution that $\tilde{\rho}$ is a gradient flow solution.

Conversely, we suppose that $\tilde{\rho}$ is a gradient flow solution, we have that $\rho \in C([0,T],\mathcal{P}_2(\R) \times \mathcal{P}_2(\R))$ and $\rho$ verifies
\eqref{distributional_eq}--\eqref{macro_velocity}. By uniqueness of the solution in Theorem~\ref{Existence}, we deduce that $\rho=\tilde{\rho}$.
\end{proof}

\section{Convergence for the kinetic model}
The convergence of the kinetic model \eqref{eq_cin_1d} towards the aggregation model is analysed in this section.
\begin{proof}[Proof of Theorem \ref{theo1.3}]
 From the assumption $\chi^{\alpha}(\theta_1+\theta_2)<1$ for $\alpha=1,2$, we obtain that $T_{\alpha}[S]$ defined in \eqref{Turning} is positive. 
Since $T_{\alpha}[S]$ is a bounded and Lipschitz continuous function, we get the global in time existence of solutions to
\eqref{eq_cin_1d} and we have that $f_{\alpha}^{\varepsilon}\in C([0,T],L^{\infty} \cap L_1^{+}(\R))$ and $\int x^2 f_{\alpha}^{\varepsilon} dx  <\infty$.\\
To prove the convergence result stated in Theorem \ref{theo1.3}, we consider the zeroth and first order moments of the distribution $f_\alpha ^{\varepsilon}(x,v,t)$ introduced previously.
\begin{equation*}
\rho_{\alpha}^{\varepsilon}:=f_{\alpha}^{\varepsilon}(1)+f_\alpha ^{\varepsilon}(-1),\quad J_{\alpha}^{\varepsilon}:=(f_{\alpha}^{\varepsilon}(1)-f_\alpha ^{\varepsilon}(-1)),\quad \text{for }\alpha=1,2.
\end{equation*}
From \eqref{eq_cin_1d}, these moments satisfy the following equations
\begin{equation}\label{sp}
 \begin{aligned}
 & \partial_{t} \rho_{\alpha}^{\varepsilon} + \partial_{x}J_{\alpha}^{\varepsilon}=0,\\
 & \partial_{t} J_{\alpha}^{\varepsilon} +\partial_{x}\rho_{\alpha}^{\varepsilon}=\frac{2}{\varepsilon} (\chi_{\alpha}\partial_{x}S^{\varepsilon}\rho_{\alpha}^{\varepsilon}- J_{\alpha}^{\varepsilon}),\quad \text{for }\alpha=1,2.
 \end{aligned}
\end{equation}
From the first equation of \eqref{sp}, we deduce that $\forall t \in [0,T]$, $\abs{\rho_{\alpha}^{\varepsilon}}(t,\R)= \abs{\rho_{\alpha}^{ini}}(\R)$. 
Therefore, for all $ t \in [0,T]$ the sequence $(\rho_{\alpha}^{\varepsilon}(t,\cdot))_{\varepsilon}$ is relatively compact in $\mathcal{M}_{b}(\mathbb{R}) - \sigma(\mathcal{M}_{b}(\mathbb{R}), C_{0}^0(\mathbb{R}))$. 
Since $J_{\alpha}^{\varepsilon}$ is uniformly bounded in $C^0([0,T],L^1(\R))$, using the same token as in the proof of the existence, there exists $\rho_{\alpha} \in \mathcal{S}_\mathcal{M}$  such that
\begin{equation*}
 \rho_{\alpha}^{\varepsilon} \rightharpoonup  \rho_{\alpha} \quad \text{weakly in } \mathcal{S}_\mathcal{M} ,\quad \text{for } \alpha=1, 2. 
\end{equation*}
From the second equation of \eqref{sp}, we have 
\begin{equation*}
\begin{aligned}
  J_{\alpha}^{\varepsilon}&=\chi_{\alpha}\partial_{x}S^{\varepsilon}\rho_{\alpha}^{\varepsilon}-\frac{\varepsilon}{2} \left(\partial_{t} J_{\alpha}^{\varepsilon} + \partial_{x}\rho_{\alpha}^{\varepsilon}\right),\quad \text{in the distributional sense} \\
                   &:=A^{\varepsilon}+R^{\varepsilon}.
\end{aligned}
\end{equation*}
We have that $R^{\varepsilon}$ converges weakly to zero in  the sense of distributions. From Lemma~\ref{technical_lemma}, one obtains
\begin{equation*}
 \int_0^T \int_\R \hat{a} (\theta_1 \rho_1^{\varepsilon}+\theta_2 \rho_2^{\varepsilon})\phi(t,x)\rho_{\alpha}^{\varepsilon}(t,dx)dt \rightarrow \int_0^T\int_{\R} \hat{a}(\theta_1\rho_1+\theta_2\rho_2)\phi(t,x)\rho_{\alpha}(t,dx)dt.
\end{equation*}
We conclude that 
\begin{equation*}
 J_{\alpha}^{\varepsilon} \rightharpoonup \chi_{\alpha} \hat{a}(\theta_1 \rho_1+\theta_2 \rho_2) \rho_{\alpha} \quad \text{in the sense of distributions}.
\end{equation*}
Passing to the limit in the first equation of \eqref{sp}, we deduce that $\rho_{\alpha}$ satisfies \eqref{distributional_eq} in the sense of distributions. We use uniqueness of duality solutions to conclude the proof.
\end{proof}

\section{Numerical simulations}
This section is devoted to the numerical simulation of system \eqref{distributional_eq}. We provide a numerical scheme which preserves basic properties of the system such as 
positivity, conservation of mass for each species and conservation of the weighted center of mass. Moreover, we prove the convergence of the numerical approximation towards the duality solution defined in Theorem~\ref{Existence}.

\subsection{Numerical scheme and properties}
Let us consider a cartesian grid of time step $\Delta t$ and space step $\Delta x$. We denote $x_j=j\Delta x, j \in \Z,t^n=n\Delta t,n\in \N$.
An approximation of $\rho_{\alpha}(t^n,x_j)$ denoted $\rho_{\alpha,j}^n$ is computed by using a finite volume approach where the flux $F_{\alpha,j-1/2}^n$ is given by the flux vector splitting method (see \cite{James}).
Assuming that $(\rho_{\alpha,j}^n)$ are known at time $t^n$, we compute $\rho_{\alpha,j}^{n+1}$ by the scheme:
\begin{equation}\label{numerical_scheme}
\left \{
 \begin{aligned}
 &  \frac{\rho_{\alpha,j}^{n+1}-\rho_{\alpha,j}^n}{\Delta t}+\frac{F_{\alpha,j+1/2}^n-F_{\alpha,j-1/2}^n}{\Delta x}=0\quad \text{for }\alpha=1, 2 \text{ and } j\in \Z,\\
 &  F_{\alpha,j-1/2}^n=(\hat{a}_{j-1}^n)^{+}\rho_{\alpha, j-1}^n+(\hat{a}_j^n)^{-}\rho_{\alpha,j}^n,\\
 & \hat{a}_j^n=\sum_{i \neq j} \partial_x K(x_j-x_i)\left(\theta_1\rho_{1,i}^n+\theta_2\rho_{2,i}^n\right),
 \end{aligned}
 \right.
\end{equation}
where $(\cdot)^{+} := \max\{(\cdot), 0\}$,  $(\cdot)^{-} := \min\{(\cdot), 0\}$ are respectively the positive and negative part of $(\cdot)$. Then we reconstruct 
$$
\rho_{\alpha, \Delta x}(t,x)=\sum_{n=0}^{N_t-1} \sum_{j\in \Z}\rho_{\alpha,j}^n \mathds{1}_{[t^n,t^{n+1}[}(t)\delta_{x_j}(x),
$$
where $\delta_{x_j}$ is the Dirac delta function at $x_j=j\Delta x$.
We first verify that this scheme allows the conservation of the mass and of the weighted center of mass.
\begin{proposition}
Let us consider $(\rho_{1}^{ini},\rho_2^{ini}) \in \mathcal{P}_2(\R)^2$ such that for $\alpha=1,2$, $\rho_{\alpha}^{ini}=\sum_{j\in\Z} \rho_{\alpha,j}^{0} \delta_{x_j}$.
We assume that for $n \in \N^*$, $(\rho_{\alpha,j}^n)_{j,n}$ are given by the numerical scheme \eqref{numerical_scheme}. Then the conservation of the mass of each species and of the weighted center of mass hold:
\begin{equation}\label{mass_conservation}
\forall n\in \N,\quad \sum_{j \in \mathbb{Z}} \rho_{\alpha,j}^{n+1}=\sum_{j \in \mathbb{Z}} \rho_{\alpha,j}^{n}\quad \text{for } \alpha=1, 2,
\end{equation}
\begin{equation}\label{center_conservation}
\forall n \in \N,\quad \frac{\theta_1}{\chi_1} \sum_{j \in \mathbb{Z}}x_j\rho_{1,j}^{n+1}+\frac{\theta_2}{\chi_2} \sum_{j \in \mathbb{Z}}x_j\rho_{2,j}^{n+1}=\frac{\theta_1}{\chi_1} \sum_{j \in \mathbb{Z}}x_j\rho_{1,j}^{n}+\frac{\theta_2}{\chi_2} \sum_{j \in \mathbb{Z}}x_j\rho_{2,j}^{n}.
\end{equation}
\end{proposition}
\begin{proof}
 Identity \eqref{mass_conservation} can be obtained directly by summing over $j \in \Z$ the first equation in  \eqref{numerical_scheme}.

We now show \eqref{center_conservation}. Multiplying by $x_j$ the first equation in \eqref{numerical_scheme} and summing over $j\in \Z$, one gets 
\begin{multline*}
  \frac{1}{\chi_{\alpha}} \sum_{j \in \mathbb{Z}}x_j\rho_{\alpha,j}^{n+1}=\frac{1}{\chi_{\alpha}} \sum_{j \in \mathbb{Z}}x_j\rho_{\alpha,j}^{n}-\frac{\Delta t}{\Delta x} \sum_{j \in \mathbb{Z}}x_j(\hat{a}_{j}^n)^{+}\rho_{\alpha,j}^n
  +\frac{\Delta t}{\Delta x} \sum_{j \in \mathbb{Z}}x_j(\hat{a}_{j-1}^n)^{+}\rho_{\alpha,j-1}^n\\
  +\frac{\Delta t}{\Delta x} \sum_{j \in \mathbb{Z}}x_j(\hat{a}_{j}^n)^{-}\rho_{\alpha,j}^n-\frac{\Delta t}{\Delta x} \sum_{j \in \mathbb{Z}}x_j(\hat{a}_{j+1}^n)^{-}\rho_{\alpha,j+1}^n,\quad \text{for } \alpha=1,2.
\end{multline*}
Using a discrete integration by parts, one gets
\begin{equation*}
 \frac{1}{\chi_{\alpha}} \sum_{j \in \mathbb{Z}}x_j\rho_{\alpha,j}^{n+1}=\frac{1}{\chi_{\alpha}} \sum_{j \in \mathbb{Z }}x_j\rho_{\alpha,j}^{n}+\Delta t \sum_{j \in \mathbb{Z}}\left((\hat{a}_{j}^n)^{+}+(\hat{a}_{j}^n)^{-}\right)\rho_{\alpha,j}^n=\frac{1}{\chi_{\alpha}} \sum_{j \in \mathbb{Z }}x_j\rho_{\alpha,j}^{n}+\Delta t \sum_{j \in \mathbb{Z}}\hat{a}_{j}^n \rho_{\alpha,j}^n.
\end{equation*}
Finally, we get 
\begin{equation*}
 \frac{\theta_1}{\chi_1} \sum_{j \in \mathbb{Z}}x_j\rho_{1,j}^{n+1}+ \frac{\theta_2}{\chi_2} \sum_{j \in \mathbb{Z}}x_j\rho_{2,j}^{n+1}= \frac{\theta_1}{\chi_1} \sum_{j \in \mathbb{Z}}x_j\rho_{1,j}^{n}+ \frac{\theta_2}{\chi_2} \sum_{j \in \mathbb{Z}}x_j\rho_{2,j}^{n}+\Delta t \sum_{j \in \mathbb{Z}}\hat{a}_{j}^n\left(\theta_1\rho_{1,j}^n+\theta_2\rho_{2,j}^n\right).
\end{equation*}
From \eqref{numerical_scheme}, we have that 
\begin{equation*}
 \sum_{j \in \mathbb{Z}}\hat{a}_{j}^n\left(\theta_1\rho_{1,j}^n+\theta_2\rho_{2,j}^n\right)=\sum_{ i\neq j} \partial_x K(x_j-x_i)\left(\theta_1\rho_{1,j}^n+\theta_2\rho_{2,j}^n\right)\left(\theta_1\rho_{1,i}^n+ \theta_2\rho_{2,i}^n\right).
\end{equation*}
Swapping indices $i$ and $j$ and using the oddness of $\partial_x K$ yields
\begin{equation*}
 \sum_{j \in \mathbb{Z}}\hat{a}_{j}^n\left(\theta_1\rho_{1,j}^n+\theta_2\rho_{2,j}^n\right)=0.
\end{equation*}
Then \eqref{center_conservation} follows.
\end{proof}

\begin{lemma}\label{positivity}
Let $(\rho_{1}^{ini},\rho_{2}^{ini})$ be in $\mathcal{P}_2(\R)^2$ such that $\rho_{\alpha}^{ini}=\sum_{j\in\Z} \rho_{\alpha,j}^{0} \delta_{x_j}$ with $\sum_{j\in\Z} \rho_{\alpha,j}^{0}=1$ and $\rho_{\alpha,j}\geq 0$ for $\alpha=1,2$. Assuming that for $n\in \N^*$, $(\rho_{\alpha,j}^{n})_{j,n}$ are given by the numerical scheme \eqref{numerical_scheme}.
If the following CFL condition holds
\begin{equation}\label{CFL}
\|\partial_x K\|_{L^{\infty}}(\theta_1+\theta_2) \frac{\Delta t}{\Delta x}<1,
\end{equation}
Then for all $n\in \N$, $\rho_{\alpha,j}^n\geq 0$ and we have $\sup_{j,n} \abs{\hat{a}_j^n} \leq \|\partial_x K\|_{L^{\infty}}(\theta_1+\theta_2)$.
\end{lemma}
\begin{proof}
This result is proved by induction. Let us assume that at time $n$, for all $j\in \Z$, $\rho_{\alpha,j}^n$ is positive and $\sup_{j} \abs{\hat{a}_j^n} \leq \|\partial_x K\|_{L^{\infty}}(\theta_1+\theta_2)$.
From \eqref{numerical_scheme}, it follows that
\begin{equation}\label{eq_rho_n}
 \rho_{\alpha,j}^{n+1}=\left(1-\frac{\Delta t}{\Delta x}\abs{\hat{a}_j^n}\right)\rho_{\alpha,j}^n +\frac{\Delta t}{\Delta x} (\hat{a}_{j-1}^n)^{+}\rho_{\alpha, j-1}^n-\frac{\Delta t}{\Delta x}(\hat{a}_{j+1}^n)^{-}\rho_{\alpha, j+1}^n.
\end{equation}
Using the condition \eqref{CFL} and the fact that $\sup_{j} \abs{\hat{a}_j^n} \leq \|\partial_x K\|_{L^{\infty}}(\theta_1+\theta_2)$, we get that $\frac{\Delta t}{\Delta x}\abs{\hat{a}_j^n}<1$.
Therefore $\rho_{\alpha,j}^{n+1}$ is positive as a linear combinaison of positive numbers.

Then, recalling the expression of $\hat{a}_j^n$ given in \eqref{numerical_scheme}, using the fact that $\rho_{\alpha,j}^{n+1}$, $j\in \Z$, are positive and the conservation of the mass, \eqref{mass_conservation} yields
\begin{equation*}
 \abs{\hat{a}_j^{n+1}} \leq \|\partial_x K\|_{L^{\infty}}(\theta_1+\theta_2).
\end{equation*}
\end{proof}

\subsection{Convergence of the numerical solution to the theoretical solution}
In this part, we prove that the numerical scheme given in \eqref{numerical_scheme} converges to the duality solution obtained in Theorem~\ref{Existence}.
\begin{theorem}[Convergence of the numerical scheme]\label{convergence_numerical_scheme}
Let $T>0$, $\Delta x>0$ and $\Delta t>0$ such that \eqref{CFL} is satisfied and denote $N_t=\frac{T}{\Delta t}$.
Let $\rho_\alpha^{ini}\in \mathcal{P}_2(\R)$, we define
\begin{equation*}
 \rho_{\alpha,j}^0 = \int_{x_{j-\frac 12}}^{x_{j+\frac 12}} \rho_\alpha^{ini}(x)\,dx, \qquad j\in \Z.
\end{equation*}
Let us define $\rho_{\alpha, \Delta x} \in \mathcal{M}_b([0,T] \times \mathbb{R})$ by
 \begin{equation}\label{f11-7}
\rho_{\alpha, \Delta x}(t,x)=\sum_{n=0}^{N_t-1} \sum_{j\in \Z}\rho_{\alpha,j}^n \mathds{1}_{[t^n,t^{n+1}[}(t)\delta_{x_j}(x),
\end{equation}
where $(\rho_{\alpha,j}^{n})_{j,n}$ computed by \eqref{numerical_scheme}.\\
Then, we have
\begin{equation*}
 \rho_{\alpha, \Delta x} \rightharpoonup \rho_{\alpha}\quad \text{weakly in}\quad \mathcal{M}_b([0,T] \times \mathbb{R}) \quad \text{as} \quad \Delta x \rightarrow 0, 
\end{equation*}
where $\rho_{\alpha}$ is the duality unique solution of Theorem  \ref{Existence} 
with initial data $\rho_\alpha^{ini}$.
\end{theorem}
\begin{proof}[Proof of Theorem~\ref{convergence_numerical_scheme}]
For the initial data, it is clear that when $\Delta x\to 0$, we have $\rho_{\alpha,\Delta x}(t=0) \rightharpoonup \rho_\alpha^{ini}$ weakly.
From Lemma~\ref{positivity}, we get that for all $j\in \Z,n\in \N$, values of $\rho_{\alpha,j}^n$ are positive.\\

\textbf{Step 1}: Extraction of a convergent subsequence\\
Equation \eqref{mass_conservation} implies that the total variation of $\rho_{\alpha, \Delta x}$ is fixed and independant of $\Delta x$.
\begin{equation*}
\abs{\rho_{\alpha, \Delta x}}([0,T] \times \R)=T\sum_{j \in \Z}\rho_{\alpha,j}^{ini}.
\end{equation*}
Therefore, there exists a subsequence of $\rho_{\alpha,\Delta x}$ that converges weakly to $\rho_{\alpha} \in \mathcal{M}_b([0,T] \times \mathbb{R})$.\\

\textbf{Step 2}: Modified equation satisfied by $\rho_{\alpha, \Delta x}$\\
Let be $\phi \in C^{\infty}_c((0,T)\times \R)$. From the definition of $\rho_{\alpha, \Delta x}$ in \eqref{f11-7}, we have
\begin{equation*}
 <\partial_t \rho_{\alpha, \Delta x},\phi>=-\int_{[0,T]\times \R}\rho_{\alpha, \Delta x} \partial_t \phi=-\sum_{n=0}^{N_t-1}\sum_{j \in \Z}\rho_{\alpha,j}^n \int_{t^n}^{t^{n+1}}\partial_t \phi(x_j,t) dt,
\end{equation*}
Here and below we use $<\cdot, \cdot>$ to denote the dual product in the sense of distributions. Discrete integration by parts yields
\begin{equation*}
  <\partial_t \rho_{\alpha, \Delta x},\phi>=-\sum_{n=0}^{N_t-1}\sum_{j\in \Z}\rho_{\alpha,j}^n(\phi_j^{n+1}-\phi_j^{n})=\sum_{n=1}^{N_t}\sum_{j\in \Z}(\rho_{\alpha,j}^{n}-\rho_{\alpha,j}^{n-1})\phi_j^n,
\end{equation*}
where we use the notation $\phi_j^n:=\phi(x_j,t^n)$.
Using \eqref{eq_rho_n} and applying transformations to indices yields
$$
  <\partial_t \rho_{\alpha, \Delta x},\phi>=\frac{\Delta t}{\Delta x}\sum_{n=0}^{N_t-1}\sum_{j\in \Z}(\hat{a}_{j}^{n})^{+}\rho_{\alpha,j}^{n}(\phi_{j+1}^{n+1}-\phi_{j}^{n+1})
+\frac{\Delta t}{\Delta x} \sum_{n=0}^{N_t-1}\sum_{j\in \Z}(\hat{a}_{j}^{n})^{-}\rho_{\alpha,j}^{n}(\phi_{j}^{n+1}-\phi_{j-1}^{n+1}).
$$
Taylor expansions gives the existence of $\zeta^j$ in $(x_j,x_{j+1})$ and $\hat{\zeta}^j$ in $(x_{j-1},x_j)$ such that
\begin{equation*}
 \begin{aligned}
&\phi_{j+1}^{n+1}=\phi_j^{n+1} + \Delta x \partial_x \phi_j^{n+1} + \frac{(\Delta x)^2}{2} \partial_{xx} \phi(\zeta^j,t^{n+1}),\\
& \phi_{j-1}^{n+1}=\phi_j^{n+1} -\Delta x \partial_x \phi_j^{n+1} + \frac{(\Delta x)^2}{2} \partial_{xx} \phi(\hat{\zeta}^j,t^{n+1}).
 \end{aligned}
\end{equation*}
Putting together, one obtains
\begin{equation*}
   <\partial_t \rho_{\alpha, \Delta x},\phi>=\Delta t\sum_{n=0}^{N_t-1}\sum_{j\in \Z}\hat{a}_{j}^{n}\rho_{\alpha,j}^{n}\partial_x \phi_j^{n+1}+ R_{\alpha}^1(\Delta x,\Delta t),
\end{equation*}
where $R_{\alpha}^1(\Delta x,\Delta t)$ is given by
\begin{equation*}
 \begin{aligned}
  R_{\alpha}^1(\Delta x,\Delta t):=&\frac{\Delta t}{\Delta x}\sum_{n=0}^{N_t-1}\sum_{j\in \Z}(\hat{a}_{j}^{n})^{+}\rho_{\alpha,j}^{n}\left(\frac{(\Delta x)^2}{2} \partial_{xx} \phi(\zeta^j,t^{n+1})\right)\\
  -&\frac{\Delta t}{\Delta x} \sum_{n=0}^{N_t-1}\sum_{j\in \Z}(\hat{a}_{j}^{n})^{-}\rho_{\alpha,j}^{n}\left(\frac{(\Delta x)^2}{2} \partial_{xx} \phi(\hat{\zeta}^j,t^{n+1})\right).
 \end{aligned}
\end{equation*}
From \eqref{f11-7} and the definition of $\hat{a}$ in \eqref{macro_velocity}, we have 
\begin{equation*}
 \hat{a}(\theta_1 \rho_{1,\Delta x}+\theta_2 \rho_{2,\Delta x})=\sum_{n=0}^{N_t-1}\sum_{j\in \Z} \hat{a}_j^n \mathds{1}_{[t^n,t^{n+1}[}(t)\delta_{x_j}(x).
\end{equation*}
where $\hat{a}_j^n$ are defined in \eqref{numerical_scheme}.
We get that
\begin{equation*}
 <\hat{a}(\theta_1 \rho_{1,\Delta x}+\theta_2 \rho_{2,\Delta x}) \rho_{\alpha, \Delta x},\partial_x \phi>=-\sum_{n=0}^{N_t-1}\sum_{j\in \Z} \hat{a}_j^n\rho_{\alpha,j}^n \int_{t^n}^{t^{n+1}}\partial_x \phi(x_j,t)dt.
\end{equation*}
From the Taylor expansion of $\partial_x \phi(x_j,t)$:
\begin{equation*}
 \partial_x \phi(x_j,t)=\partial_x \phi(x_j,t^{n+1})+(t-t^{n+1}) \partial_{xt} \phi(x_j,\tau_t^{n}),
\end{equation*}
with $\tau_t^{n} \in (t,t^{n+1})$, one sees that
\begin{equation*}
<\hat{a}(\theta_1 \rho_{1,\Delta x}+\theta_2 \rho_{2,\Delta x}) \rho_{\alpha, \Delta x},\partial_x \phi>=-\Delta t \sum_{n=0}^{N_t-1}\sum_{j\in \Z} \hat{a}_j^n\rho_{\alpha,j}^n \partial_x \phi_j^{n+1}+R_{\alpha}^2(\Delta x, \Delta t),
\end{equation*}
where $R_{\alpha}^2(\Delta x,\Delta t)$ is defined as follows: 
\begin{equation*}
R_{\alpha}^2(\Delta x, \Delta t):=-\sum_{n=0}^{N_t-1}\sum_{j\in \Z} \hat{a}_j^n\rho_{\alpha,j}^n \int_{t^n}^{t^{n+1}}(t-t^{n+1}) \partial_{xt} \phi(x_j,\tau_t^{n})dt.
\end{equation*}
The modified equation satisfied by $\rho_{\alpha,\Delta x}$ in the distributional sense writes:
\begin{equation*}
 \int_{0}^T\int_{\R}\rho_{\alpha, \Delta x} \partial_t \phi(t,x)+\int_{0}^{T}\int_{\R} \hat{a}(\theta_1 \rho_{1,\Delta x}+\theta_2 \rho_{2,\Delta x}) \rho_{\alpha, \Delta x} \partial_x \phi=R_{\alpha}^1(\Delta x,\Delta t)+R_{\alpha}^2(\Delta x, \Delta t).
\end{equation*}
From Lemma \ref{positivity}, we deduce that the terms $R_{\alpha}^1$ and $R_{\alpha}^2$ satisfy the estimates:
\begin{equation*}
 \begin{aligned}
  & \abs{R_{\alpha}^1}\leq C T \Delta x \|\partial_{xx} \phi\|_{L^{\infty}}, \qquad
  & \abs{R_{\alpha}^2} \leq C T \Delta x \|\partial_{tx} \phi\|_{L^{\infty}}, 
 \end{aligned}
\end{equation*}
where $C$ stands for a nonnegative constant.
Passing to the limit and using the technical Lemma \ref{technical_lemma}, we conclude that the limit $\rho_{\alpha}$ satisfies \eqref{distributional_eq} in the distributional sense with 
the expression \eqref{macro_velocity} for the velocity. By uniqueness result in Theorem~\ref{Existence}, we deduce that $\rho_{\alpha}$ is the unique duality solution of \eqref{aggregation_eq}.
\end{proof}

\subsection{Dynamics of aggregates and numerical simulations}
In this part, we carry out simulations of Equation \eqref{distributional_eq} obtained thanks to scheme \eqref{numerical_scheme}.
Before numerically simulating the hydrodynamic behavior of the chemotaxis model, we first clarify the aggregate dynamics of this model, especially on the synchronising dynamics between aggregates of different species.   

For the sake of simplicity, we choose $\theta_1=\theta_2=1$ and $\displaystyle K=\frac{1}{2}e^{-\abs{x}}$ in \eqref{distributional_eq}, 
which corresponds to the particular case of bacterial chemotaxis (see \eqref{Kchemo}).
To illustrate the synchronising dynamics of the aggregates for \eqref{distributional_eq},  we consider  the initial data given by sums of aggregates
\begin{equation*}
  \rho_1^0=\sum_k \mu_k \delta_{x_k^0}, \quad \rho_2^0=\sum_k \nu_k \delta_{y_k^0},
\end{equation*}
and look for a solution in the form
\begin{equation*}
 \rho_1(t,x)=\sum_k \mu_k \delta_{x_k(t)}, \quad \rho_2(t,x)=\sum_k \nu_k \delta_{y_k(t)}.
\end{equation*}
 We denote by $u_1$ and $u_2$ antiderivatives of $\rho_1$ and $\rho_2$, respectively. Then the equation \eqref{distributional_eq} reads
\begin{equation}\label{equ}
\partial_t u_\alpha + \chi_\alpha \hat{a} \rho_\alpha = 0, \qquad \alpha = 1,2,
\end{equation}
in the sense of distributions. Direct computation shows that
\begin{equation*}
 \begin{aligned}
  & \hat{a}\rho_1 = \sum_{k,\ell} \mu_k(\mu_\ell \widehat{\partial_x K}(x_k-x_\ell) + \nu_\ell \widehat{\partial_x K}(x_k-y_\ell))\delta_{x_k},\\
  & \hat{a}\rho_2 = \sum_{k,\ell} \nu_k(\mu_\ell \widehat{\partial_x K}(y_k-x_\ell) + \nu_\ell \widehat{\partial_x K}(y_k-y_\ell))\delta_{y_k}.
 \end{aligned}
\end{equation*}
Injecting these expressions into equation \eqref{equ}, the positions $x_k$ and $y_k$   satisfy the system of ODEs
\begin{equation*}
 \begin{aligned}
x'_k(t) &= \chi_1 \sum_{\ell} (\mu_\ell \widehat{\partial_x K}(x_k-x_\ell) + \nu_\ell \widehat{\partial_x K}(x_k-y_\ell)),\\
y'_k(t) &= \chi_2 \sum_{\ell} (\mu_\ell \widehat{\partial_x K}(y_k-x_\ell) + \nu_\ell \widehat{\partial_x K}(y_k-y_\ell)).
 \end{aligned}
\end{equation*}
We recover the same system for particle solutions as in DiFrancesco and Fagioli \cite{Francesco} for two species. 
See also similar aggregate dynamics for single species in \cite{Carrillo,jamesnv}.  In the case of one single species,  the system of ODEs is determinant before any collision of  aggregates, and after each collision, one can always `restart' the particle system till final collapse of all aggregates.  
However, the case of collisions between particles of different species is more complex, since it does not necessarily imply whether the particles of different species will synchronise  or not after colliding. 
In fact, as observed in the following simulations, both `synchronising'(colliding particles of different species staying together) and `non-synchronising' cases can occur, and the transitions between the synchronising types may happen,  depending on the weighted attraction of other aggregates acting on them.

For illustration,  we assume that two points of different species collide at a time $t_0$. For instance, take $x_k(t_0)=y_k(t_0)$ for some $k$, then at this time $t_0$ we have
\begin{equation}\label{f12-1}
x'_k(t_0) = \chi_1 \gamma_k(t_0), \qquad y'_k(t_0)= \chi_2 \gamma_k(t_0), \qquad
\gamma_k=\sum_{\ell,k} (\mu_\ell \widehat{\partial_x K}(x_k-x_\ell) + \nu_\ell \widehat{\partial_x K}(x_k-y_\ell)).
\end{equation}
Note that  $\gamma_k$ characterises  external weighted attraction on $\nu_k$ and $\mu_k$, 
depending on chemo-sensitivities,  distances to other aggregates and the masses of all other aggregates.

Thus if $\chi_1\neq \chi_2$ the velocity of species 1 and 2 is not the same at this time $t_0$. 
However, with the special case at hand, $K(x)=\frac{1}{2}e^{-|x|}$, we have 
$\partial_x K(x_k-y_k)\to \frac 12$ when $x_k\xrightarrow{<} y_k$; and $\partial_x K(x_k-y_k)\to -\frac 12$ when $x_k\xrightarrow{>} y_k$.
We deduce that when $x_k<y_k$ and $x_k\to y_k$ we have
\begin{equation*}
 (y_k-x_k)'(t) = -\frac 12 (\chi_1\nu_k + \chi_2\mu_k) + (\chi_2-\chi_1)\gamma_k(t).
\end{equation*}
Obviously, in this case, particles $\mu_k$ and $\nu_k$ stay together if $(y_k-x_k)'(t)\leq 0$.
On the other hand, when  $y_k<x_k$ and $y_k\to x_k$ we have
\begin{equation*}
 (x_k-y_k)'(t) = -\frac 12 (\chi_1\nu_k + \chi_2\mu_k) + (\chi_1-\chi_2)\gamma_k(t).
\end{equation*}
In this case, particles $\mu_k$ and $\nu_k$ stay together when  $(x_k-y_k)'(t)\leq 0$. Finally, to keep $x_k(t)=y_k(t)$ for $t\geq t_0$, we need the condition 
\begin{equation}\label{f12-2}
|(\chi_1-\chi_2)\gamma_k(t)| \leq \frac 12 (\chi_1\nu_k + \chi_2\mu_k), 
\end{equation}
where $\gamma_k(t)$ is defined in \eqref{f12-1}. 
This relation characterises the weighted attraction of other aggregates acting on them. 
If the  external weighted attraction on $\nu_k$ and $\mu_k$ 
(the left hand side of \eqref{f12-2}) is small,  they will stay together. 
When the external weighted attraction is big, the attraction between $\nu_k$ and $\mu_k$ 
is relatively weak and they will move separately,
the one with bigger  motility will move faster than the other.  

We call \eqref{f12-2} the synchronising condition for $\mu_k$ and $\nu_k$. Similarly, 
we can get the  synchronising condition for any $\mu_i$ and $\nu_j$, $\forall i, j $. 
If more than two aggregates collide simultaneously, we can simply replace them by two 
aggregates of each species, each aggregate accumulating the total mass of each species.

In conclusion,  according to the dynamics defined above, we can see that the initial aggregates  will collapse such that they eventually form a single aggregate of the two species. 
The final aggregate can not separate, which is similar but illustrate more complex behaviour as one species case discussed in \cite{jamesnv}.
Now we give some numerical examples showing ``synchronising'', ``non-synchronising'', transitions between ``synchronising'' and ``non-synchronising'' dynamical behaviours for the hydrodynamic model \eqref{distributional_eq}. 

\textbf{Example 1: Synchronising dynamics}.
Take the chemosensitivity constants  $\chi_1  =10$, $\chi_2 =1$ in  \eqref{distributional_eq}, and consider initial data
\begin{equation*}
\rho_{1}^{0} = 4 e^{-5000(x+0.5)^{2}} + 2 e^{-5000(x-0.5)^{2}}, \quad \rho_{2}^{0} = 2 e^{-5000(x+0.15)^{2}}. 
\end{equation*}
It corresponds to small bumps  located at position $x_{1}(0)=-0.5$, $y_{1}(0)=-0.15$, $x_{2}(0)=0.5$, with mass $\mu_{1}=4m_{0}$,  $\mu_{2}=2m_{0}$, $\nu_{1}=2m_{0}$, where $m_{0}=\int_\R e^{-5000x^{2}} dx$. 
\begin{figure}[h]
\centering
\includegraphics[width= \textwidth]{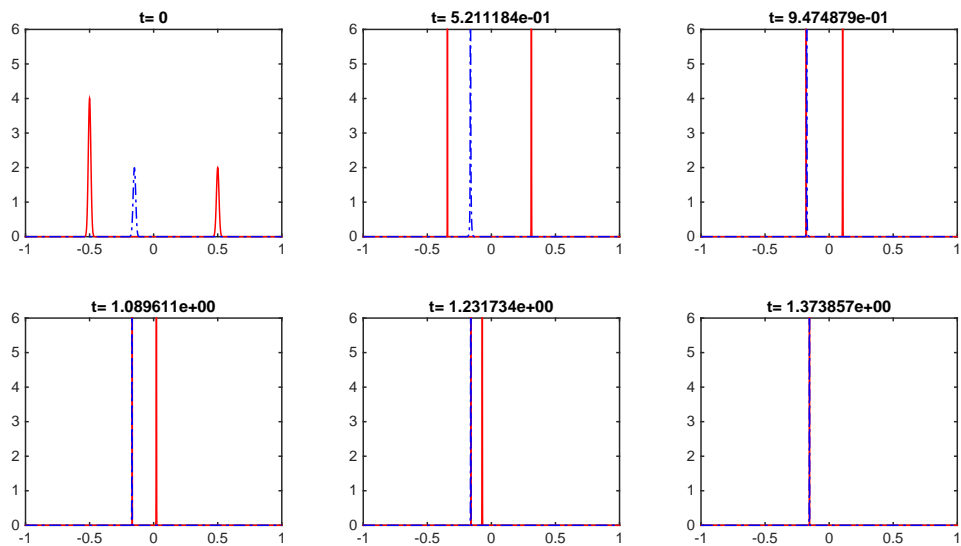}
\caption{Example 1. Snapshots of $\rho_{1}$ (red solid line) and $\rho_{2}$ (blue dashdot). The evolution shows the synchronising dynamics after first collision.}
\label{fig1}
\end{figure} 
Figure \ref{fig1} displays numerical results obtained thanks to the scheme \eqref{numerical_scheme} defined above.
We first observe the fast blow-up with the formation of Dirac deltas.
Then, the numerical simulation shows that $\mu_{1}$ and  $\nu_{1}$   collapse for the first  time at  $t_1 \approx 0.947$, with $x_1(t_1) =  y_1(t_1 )  \approx -0.18$, and  $x_2(t_1 )  \approx 0.12$. 
We check the ``synchronising condition''  \eqref{f12-2}:
\begin{equation*}
\begin{aligned}
 LHS &= |(\chi_1-\chi_2)\gamma_1(t_1)| = (10-1) \times 2 \times {1\over 2} e^{-(x_2-x_1)}= 9e^{-(x_2-x_1)} < 9 , \quad \forall x_1, x_2,\\
 RHS &= \frac 12 (\chi_1\nu_1 + \chi_2\mu_1) = \frac 12 (10\times 2 + 1\times 4) = 12.
\end{aligned}
\end{equation*}
Thus the ``synchronising condition'' \eqref{f12-2} is always satisfied, then they will move together afterwards till final collapse with $\mu_{2}$.
This evolutionary dynamics is shown in Figure \ref{fig1}. The numerical result confirms the synchronising dynamics of the aggregates.

\vspace{2mm}

\textbf{Example 2: Non-synchronising dynamics}.
Take the chemosensitivity constants $\chi_1  =10$, $\chi_2 =1$ in  \eqref{distributional_eq}, and consider initial data
\begin{equation*}
 \rho_{1}^{0} = 2 e^{-5000(x+0.5)^{2}} + 4 e^{-5000(x-0.5)^{2}}, \quad \rho_{2}^{0} = 2 e^{-5000(x+0.15)^{2}}.
\end{equation*}
It corresponds to small bumps  located at  position $x_{1}(0)=-0.5$, $y_{1}(0)=-0.15$, $x_{2}(0)=0.5$, with mass $\mu_{1}=2m_{0}$,  $\mu_{2}=4m_{0}$, $\nu_{1}=2m_{0}$, where $m_{0}=\int_\R e^{-5000x^{2}} dx$. 
The numerical simulation in Figure \ref{fig2}  shows that $\mu_{1}$ and  $\nu_{1}$   collapse for the first  time at  $t_1 \approx 0.9$, $x_1(t_1) =  y_1(t_1 )  \approx -0.15$, and  $x_2(t_1 )  \approx 0.25$. Direct computation shows that
\begin{equation*}
 \begin{aligned}
  LHS &=|(\chi_1-\chi_2)\gamma_1(t_1)| = (10-1) \times 4 \times {1\over 2} e^{-(x_2-x_1)} \approx  18 e^{-0.4} \approx 12.066,\\
  RHS &=\frac 12 (\chi_1\nu_1 + \chi_2\mu_1) = \frac 12 (10\times 2 + 1\times 2) = 11,
 \end{aligned}
\end{equation*}
thus   the ``synchronising condition'' \eqref{f12-2} is not  satisfied, then they will   change their order after intersection and travel separately. 
 The simulation shows $\mu_{1}$ will collapse with $\mu_{2}$ at time $t_2 \approx 1.61 $, and finally collapse with $\nu_{1}$ at time $t_3 \approx 1.85 $. 
 This dynamics is shown in Figure \ref{fig2}. The numerical result confirms the non-synchronising dynamics of the aggregates.
\begin{figure}[h]
\centering
\includegraphics[width= \textwidth]{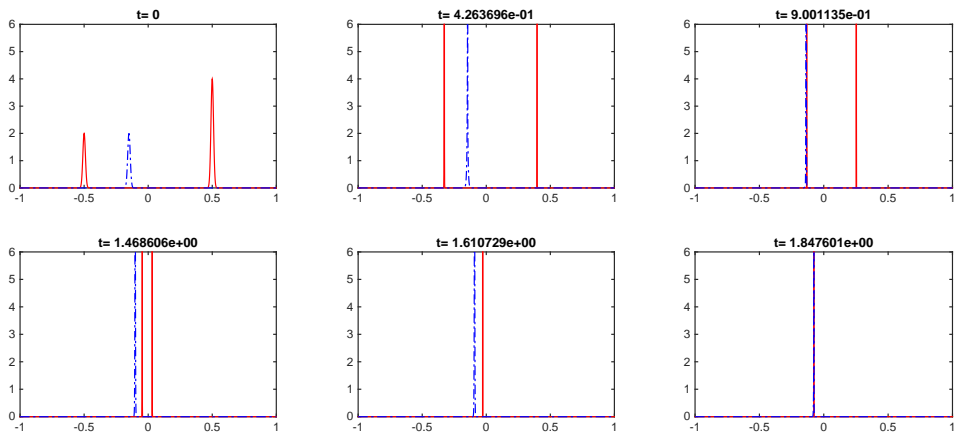}
\caption{Example 2.   Snapshots of $\rho_{1}$ (red solid line) and $\rho_{2}$ (blue dashdot). The evolution shows the non-synchronising dynamics after first collision.}
\label{fig2}
\end{figure}      
 
\textbf{Example 3: Transition from synchronising to  non-synchronising dynamics}.
Take the chemosensitivity constants $\chi_1  =10$, $\chi_2 =1$ in  \eqref{distributional_eq}, and slightly modify the  initial data of Example 2 to
\begin{equation*}
\rho_{1}^{0} = 2 e^{-5000(x+0.5)^{2}} + 4 e^{-5000(x-0.5)^{2}}, \quad \rho_{2}^{0} = 2 e^{-5000(x+0.3)^{2}}. 
\end{equation*}
It corresponds to small bumps  located at  position $x_{1}(0)=-0.5$,  $y_{1}(0)=-0.3$, $x_{2}(0)=0.5$, with mass $\mu_{1}=2m_{0}$,  $\mu_{2}=4m_{0}$, $\nu_{1}=2m_{0}$.
The numerical simulation displayed in Figure \ref{fig3} shows that $\mu_{1}$ and  $\nu_{1}$   collapse for the first  time at  $t_1 \approx 0.47$ with $x_1(t_1) =  y_1(t_1 )  \approx -0.29$, and  $x_2(t_1 )  \approx 0.39$. 
Direct computation shows that
\begin{equation*}
 \begin{aligned}
  LHS &=|(\chi_1-\chi_2)\gamma_1(t_1)| = (10-1) \times 4 \times {1\over 2} e^{-(x_2-x_1)} \approx  18 e^{-0.68} \approx 9.1191, \\
  RHS &=\frac 12 (\chi_1\nu_1 + \chi_2\mu_1) = \frac 12 (10\times 2 + 1\times 2) = 11,
 \end{aligned}
\end{equation*}
 thus   the ``synchronising condition''  \eqref{f12-2} is satisfied, then they will move together toward $\mu_2$. 
 The interesting phenomenon is that, as their distance to $\mu_2$ is decreasing, the LHS of the ``synchronising condition''  \eqref{f12-2} is increasing and finally greater than the RHS. 
 The simulation shows that, at  $t_2 \approx 1.04$, $x_1(t_2) =  y_1(t_2 )  \approx -0.26$, and  $x_2(t_2 )  \approx 0.23$, then
\begin{equation*}
 LHS=|(\chi_1-\chi_2)\gamma_1(t_1)| = (10-1) \times 4 \times {1\over 2} e^{-(x_2-x_1)} \approx  18 e^{-0.49} \approx 11=RHS, 
\end{equation*}
then after this time $t_2$, the interaction type has been changed:  the ``synchronising condition''  \eqref{f12-2} is no longer  satisfied, then they will travel separately. 
Further simulation shows that $\mu_{1}$   collapses with $\mu_{2}$ at time $t_3 \approx 2.037 $, and finally collapse with $\nu_{1}$ at time $t_4 \approx 2.32 $. 
The full dynamics is shown in Figure \ref{fig3}. The numerical result shows the transition from synchronising to  non-synchronising dynamics  of the aggregates.
\begin{figure}[h]
\centering
\includegraphics[width= \textwidth]{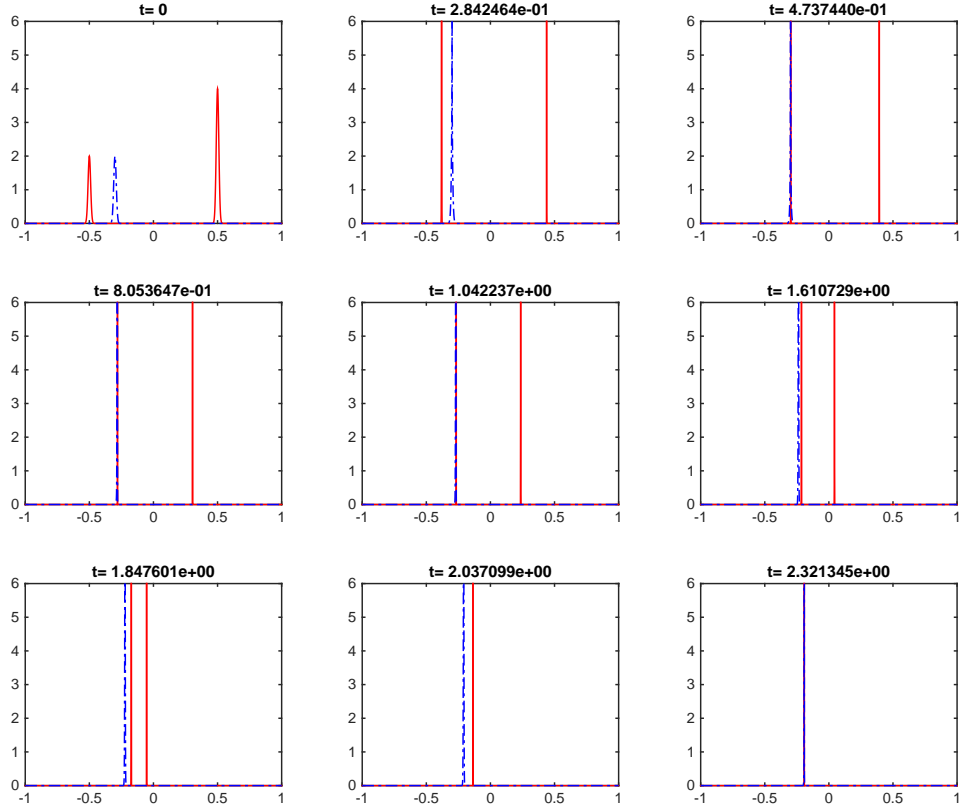}
\caption{Example 3.   Snapshots  of $\rho_{1}$ (red solid line) and $\rho_{2}$ (blue dashdot).  
From time $t_0 = 0$ to $t_1\approx0.47$, $\mu_1$ moves toward $\nu_1$. From time $t_1\approx0.47$ to $t_2 \approx 1.04$, $\mu_1$ and $\nu_1$ travel together. The synchronising type changed at  $t_2 $. After time $t_2$, $\mu_1$ overtakes $\nu_1$ and collapse with $\mu_2$ at $t_3\approx2.037$, and finally all the aggregates collapse at $t_4 \approx 2.32$.
  The evolution shows the transition from synchronising to  non-synchronising dynamics.}
\label{fig3}
\end{figure}      
 
\textbf{Example 4: More complex transition}.
Take the chemosensitivity constants  $\chi_1  =10$, $\chi_2 =1$ in  \eqref{distributional_eq}, and consider initial data
\begin{equation*}
\rho_{1}^{0} = 3 e^{-5000(x+0.8)^{2}} + 1.5 e^{-5000 (x+0.02)^{2}}, \quad \rho_{2}^{0} = 3.5 e^{-5000(x-0.02)^{2}}+ 8.5 e^{-5000(x-0.5)^{2}}. 
\end{equation*}
It corresponds to small bumps  located at  position $x_{1}(0)=-0.8$, $x_{2}(0)=-0.02$, $y_{1}(0)=0.02$, $y_{2}(0)=0.5$, with mass $\mu_{1}=3m_{0}$,  $\mu_{2}=1.5m_{0}$, $\nu_{1}=3.5m_{0}$, $\nu_{2}=8.5m_{0}$.
The snapshots of $\rho_{1}$ and $\rho_2$ are shown in Figure \ref{fig4}.
\begin{figure}[h]
\centering
\includegraphics[width= \textwidth]{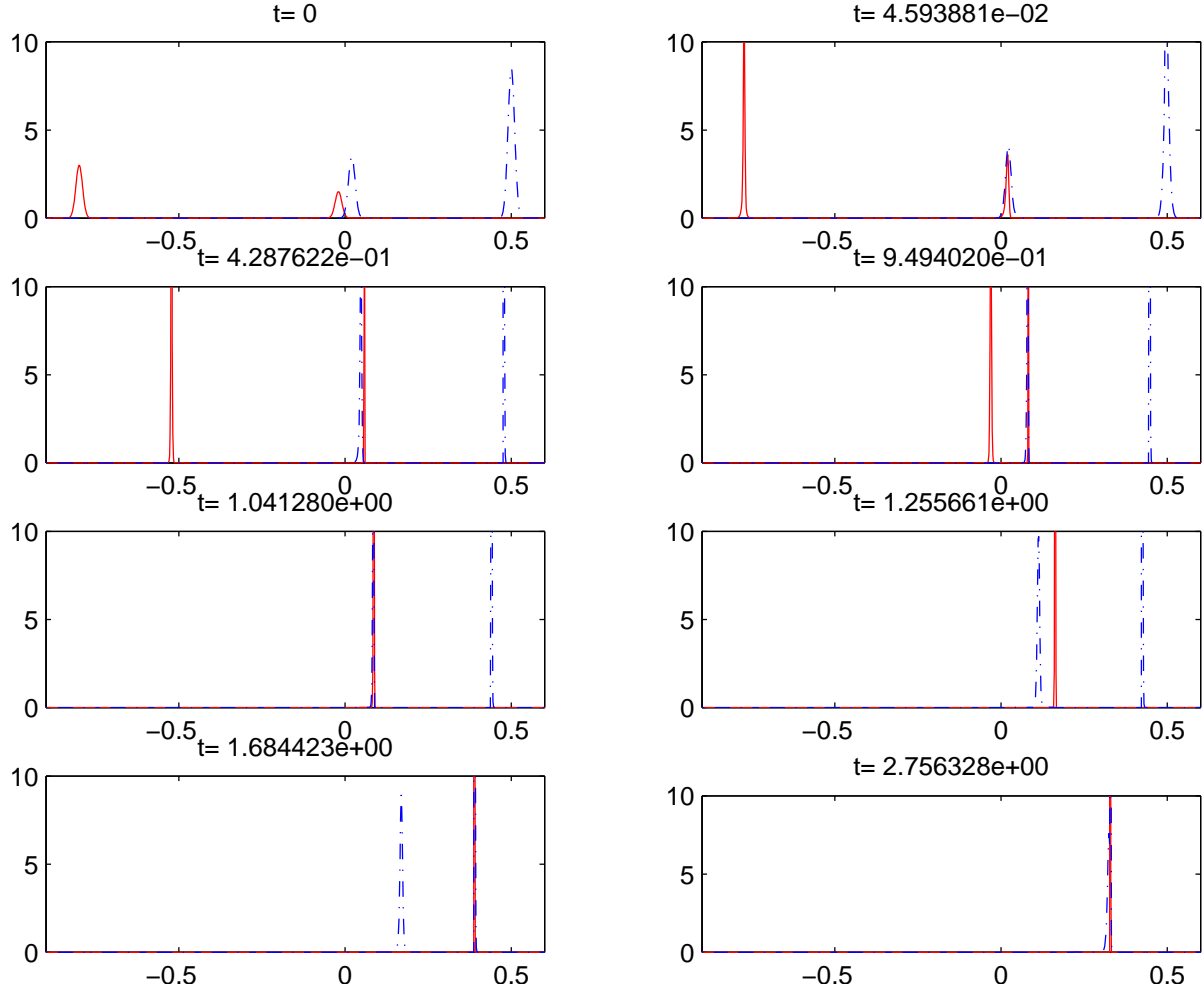}
\caption{Example 4. Snapshots of $\rho_{1}$ (red solid line) and $\rho_{2}$ (blue dashdot).}
\label{fig4}
\end{figure} 
We observe that  $\mu_2$ and $\nu_1$ meet for the first time at $t_1\approx 0.0459$ and satisfy non-synchronising condition so they separate after $t_1$. See the snapshot at $t_2\approx 0.4288$  for evidence.  
They meet for the second time at $t_3\approx 0.9494$ but the synchronising type has been changed: now they satisfy synchronising condition thus they travel together afterwards. 
At time $t_4\approx 1.04$,  $\mu_1$ catches  $\mu_2$ and $\nu_1$. Now we treat  them as $(\mu_1+\mu_2)$ and $\nu_1$: they  satisfy non-synchronising condition and separate, see snapshot at $t_5\approx 1.256$ for evidence. 
At time $t_6\approx 1.684$, $(\mu_1+\mu_2)$ collapse with  $\nu_2$, satisfying  synchronising condition and staying together till final collapse with  $\nu_1$ at $t_7\approx 2.756$. 
The illustration shows the complex changing of interaction types for the aggregate dynamics of two species chemotaxis model. 

{\bf Acknowledgement. } C.Emako and N.Vauchelet acknowledge partial support from the ANR project Kibord, 
ANR-13-BS01-0004 funded by the French Ministry of Research. J.Liao would like to acknowledge partial support  by  National Natural Science Foundation of China (No. 11301182), 
Science and Technology commission of Shanghai Municipality (No. 13ZR1453400), and a scholarship from China Scholarship Council for visiting Laboratoire Jacques-Louis Lions, UPMC, France.

\bibliographystyle{amsplain}
\bibliography{paper_references}

\providecommand{\bysame}{\leavevmode\hbox to3em{\hrulefill}\thinspace}
\providecommand{\MR}{\relax\ifhmode\unskip\space\fi MR }
\providecommand{\MRhref}[2]{%
  \href{http://www.ams.org/mathscinet-getitem?mr=#1}{#2}
}
\providecommand{\href}[2]{#2}
\begin{thebibliography}{10}

\bibitem{Ambrosio}
Luigi Ambrosio, Nicola Gigli, and Giuseppe Savar{\'e}, \emph{Gradient flows in
  metric spaces and in the space of probability measures}, second ed., Lectures
  in Mathematics ETH Z\"urich, Birkh\"auser Verlag, Basel, 2008. \MR{2401600
  (2009h:49002)}

\bibitem{Bertozzi}
Andrea~Louise Bertozzi and Jeremy Brandman, \emph{Finite-time blow-up of
  {$L^\infty$}-weak solutions of an aggregation equation}, Commun. Math. Sci.
  \textbf{8} (2010), no.~1, 45--65. \MR{2655900 (2011d:35076)}

\bibitem{Bouchut}
Fran\c{c}ois Bouchut and Fran\c{c}ois James, \emph{One-dimensional transport
  equations with discontinuous coefficients}, Nonlinear Anal. \textbf{32}
  (1998), no.~7, 891--933. \MR{1618393 (2000a:35243)}

\bibitem{JamesBouchut}
Fran{\c{c}}ois Bouchut and Fran{\c{c}}ois James, \emph{Duality solutions for
  pressureless gases, monotone scalar conservation laws, and uniqueness}, Comm.
  Partial Differential Equations \textbf{24} (1999), no.~11-12, 2173--2189.
  \MR{1720754 (2000i:35167)}

\bibitem{Carrillo}
Jose Carrillo, Marco Di~Francesco, Alessio Figalli, Thomas Laurent, and Dejan
  Slep{\v{c}}ev, \emph{Global-in-time weak measure solutions and finite-time
  aggregation for nonlocal interaction equations}, Duke Math. J. \textbf{156}
  (2011), no.~2, 229--271. \MR{2769217 (2012c:35447)}

\bibitem{CarrilloVauch}
Jos{\'e}~A. Carrillo, James Fran\c{c}ois, Fr{\'e}d{\'e}ric Lagouti{\`e}re, and
  Nicolas Vauchelet, \emph{The filippov characteristic flow for the aggregation
  equation with mildly singular potentials}, Arxiv preprint (2014).

\bibitem{Freda}
Jos\'e~Antonio Carrillo, Fr\'ed\'erique Charles, Young-Pil Choi, and Martin
  Campos-Pinto, \emph{Convergence of linearly transformed particle methods for
  the aggregation equation}, in preparation.

\bibitem{Carrillo1}
Jos{\'e}~Antonio Carrillo, Alina Chertock, and Yanghong Huang, \emph{A
  finite-volume method for nonlinear nonlocal equations with a gradient flow
  structure}, Communications in Computational Physics \textbf{17} (2015),
  233--258.

\bibitem{Bertozzi1}
Katy Craig and Andrea Bertozzi, \emph{A blob method for the aggregation
  equation}, to appear in Math. Comp.

\bibitem{Mercier}
Gianluca Crippa and Magali L\'ecureux-Mercier, \emph{Existence and uniqueness
  of measure solutions for a system of continuity equations with non-local
  flow}, Nonlinear Differential Equations and Applications NoDEA \textbf{20}
  (2013), no.~3, 523--537.

\bibitem{Francesco}
Marco Di~Francesco and Simone Fagioli, \emph{Measure solutions for non-local
  interaction {PDE}s with two species}, Nonlinearity \textbf{26} (2013),
  no.~10, 2777--2808. \MR{3105514}

\bibitem{DolakSch}
Yasmin Dolak and Christian Schmeiser, \emph{Kinetic models for chemotaxis:
  hydrodynamic limits and spatio-temporal mechanisms}, J. Math. Biol.
  \textbf{51} (2005), no.~6, 595--615. \MR{2213630 (2006k:92009)}

\bibitem{Casimir}
Casimir Emako, Charl\`ene Gayrard, Nicolas Vauchelet, Luis Neves~de Almeida,
  and Axel Buguin, \emph{Traveling pulses for a two-species chemotaxis model},
  in preparation.

\bibitem{Almeida}
Casimir Emako, Luis Neves~de Almeida, and Nicolas Vauchelet, \emph{Existence
  and diffusive limit of a two-species kinetic model of chemotaxis}, Kinetic
  and Related Models \textbf{8} (2015), no.~2, 359--380.

\bibitem{James}
Laurent Gosse and Fran{\c{c}}ois James, \emph{Numerical approximations of
  one-dimensional linear conservation equations with discontinuous
  coefficients}, Math. Comp. \textbf{69} (2000), no.~231, 987--1015.
  \MR{1670896 (2000j:65077)}

\bibitem{Helbing}
Dirk Helbing, Wenjian Yu, and Heiko Rauhut, \emph{Self-organization and
  emergence in social systems: modeling the coevolution of social environments
  and cooperative behavior}, J. Math. Sociol. \textbf{35} (2011), no.~1-3,
  177--208. \MR{2844985 (2012i:91048)}

\bibitem{GF_dual}
Fran\c{c}ois James and Nicolas Vauchelet, \emph{Equivalence between duality and
  gradient flow solutions for one-dimensional aggregation equations}, in
  preparation.

\bibitem{jamesnv}
\bysame, \emph{Chemotaxis: from kinetic equations to aggregate dynamics}, NoDEA
  Nonlinear Differential Equations Appl. \textbf{20} (2013), no.~1, 101--127.
  \MR{3011314}

\bibitem{VaucheletJ}
\bysame, \emph{Numerical methods for one-dimensional aggregation equations},
  SIAM Journal on Numerical Analysis \textbf{53} (2015), no.~2, 895--916.

\bibitem{Mittal}
Nikhil Mittal, Elena~O. Budrene, Michael~P. Brenner, and Alexander van
  Oudenaarden, \emph{Motility of escherichia coli cells in clusters formed by
  chemotactic aggregation}, Proceedings of the National Academy of Sciences
  \textbf{100} (2003), no.~23, 13259--13263.

\bibitem{OthmerAlt}
Hans Othmer, Stevens Dunbar, and Wolfgang Alt, \emph{Models of dispersal in
  biological systems}, J. Math. Biol. \textbf{26} (1988), no.~3, 263--298.
  \MR{949094 (90a:92064)}

\bibitem{Poupaud2}
Fr{\'e}d{\'e}ric Poupaud, \emph{Diagonal defect measures, adhesion dynamics and
  {E}uler equation}, Methods Appl. Anal. \textbf{9} (2002), no.~4, 533--561.
  \MR{2006604 (2004i:35259)}

\bibitem{Rachev}
Svetlozar~T. Rachev and Ludger R{\"u}schendorf, \emph{Mass transportation
  problems. {V}ol. {II}}, Probability and its Applications (New York),
  Springer-Verlag, New York, 1998, Applications. \MR{1619171 (99k:28007)}

\bibitem{VCJS}
Jonathan Saragosti, Vincent Calvez, Nikolaos Bournaveas, Axel Buguin, Pascal
  Silberzan, and Beno{\^{\i}}t Perthame, \emph{Mathematical description of
  bacterial traveling pulses}, PLoS Comput. Biol. \textbf{6} (2010), no.~8,
  e1000890, 12. \MR{2727559 (2011f:92008)}

\bibitem{Sznajd}
Katarzyna Sznajd-Weron and Jozef Sznajd, \emph{Opinion evolution in closed
  community}, International Journal of Modern Physics C \textbf{11} (2000),
  no.~06, 1157--1165.

\bibitem{Villani1}
C{\'e}dric Villani, \emph{Topics in optimal transportation}, Graduate Studies
  in Mathematics, vol.~58, American Mathematical Society, Providence, RI, 2003.
  \MR{1964483 (2004e:90003)}

\bibitem{Villani2}
\bysame, \emph{Optimal transport, old and new}, Grundlehren der Mathematischen
  Wissenschaften [Fundamental Principles of Mathematical Sciences], vol. 338,
  Springer-Verlag, Berlin, 2009.

\end{thebibliography}
\end{document}